\documentclass[reqno]{amsart}
\usepackage[english]{babel}
\usepackage{amscd,amssymb,amsmath,amsfonts,latexsym,amsthm}
\usepackage[T1]{fontenc}
\usepackage{inputenc}
\usepackage{graphicx}

\newtheorem{thm}{Theorem}[section]

\newtheorem{lem}[thm]{Lemma}
\newtheorem{prop}[thm]{Proposition}
\newtheorem{defn}[thm]{Definition}

\newtheorem{rem}[thm]{Remark}
\numberwithin{equation}{section}

\begin{document}

\title{Extensions of solvable Lie algebras with naturally graded filiform nilradical}
\author{Khudoyberdiyev A.Kh., Sheraliyeva S.A.}
\address{[A.Kh. Khudoyberdiyev]
Institute of Mathematics Academy of Sciences of Uzbekistan,  National University of Uzbekistan, Tashkent, 100174, Uzbekistan.}
\email{khabror@mail.ru}
\address{[S.A. Sheraliyeva] Institute of Mathematics Academy of Sciences of Uzbekistan,  National University of Uzbekistan, Tashkent, 100174, Uzbekistan.}
\email{abdiqodirovna@mail.ru}



\begin{abstract}
In this work we consider extensions of solvable Lie algebras with naturally graded filiform nilradicals.
Note that there exist two naturally graded filiform Lie algebras $n_{n, 1}$ and $Q_{2n}.$
We find all one-dimensional central extensions of the algebra $n_{n, 1}$ and show that any extension of
$Q_{2n}$ is split.
After that we find one-dimensional extensions of solvable Lie algebras with nilradical $n_{n, 1}$. We prove that there exists a unique non-split central extension of  solvable Lie algebras with nilradical $n_{n, 1}$ of maximal codimension. Moreover, all one-dimensional extensions of solvable Lie algebras with nilradical $n_{n, 1}$ whose codimension is equal to one are found and compared these solvable algebras with the solvable algebras with nilradicals are one-dimensional central extension of algebra $n_{n, 1}$.

\end{abstract}

\maketitle

\section{Introduction}

From Levi's theorem it is well-known that any finite-dimensional Lie algebra $L$ can be decomposed in a unique manner
into a semi-direct sum of a semi-simple Lie algebra $S$ and a solvable ideal $R$, its radical \cite{Jac}.
Thus, the investigation and classification of solvable Lie algebras is an essential step in the theory of finite-dimensional Lie algebras over the field of characteristic zero.
Moreover, the classification of finite-dimensional solvable Lie algebras was reduced to the nilpotent ones. For many years, researchers have used varios method to classify nilpotent and solvable Lie algebras and studied the problem of classification of low-dimensional Lie algebras. One of the effective methods of classifying low-dimensional nilpotent Lie algebras is the method of extensions.
There are several types of extensions, such as trivial, central, split and others. Central extension is widely used in the classification of finite-dimensional nilpotent algebras.

Central extensions are needed in physics, because the symmetry group of a quantized
system is usually a central extension of the classical symmetry group, similarly the corresponding symmetry Lie algebra of the quantum system is, in general, a central extension of the classical
symmetry algebra.  First, Skjelbred and Sund used method of central extension to obtain a classification of
nilpotent Lie algebras \cite{Sund}. But until the works \cite{degr3}, \cite{degr2}, \cite{Gong}  some researchers had suspected that this method requires very much computation and had not concentrated on this method. In \cite{degr3} an algorithm for how to use Skjelbred-Sund method is given and introduced some notations which are very suitable to use central extensions method.
Moreover, by A. Hegazi and others the analogue of Skjelbred-Sund method was presented for the Jordan and Malcev algebras \cite{ha16, hac16}.
After that in recent years central extensions method have been used to the classification of various types of nilpotent algebras,
and classification of many classes of low-dimensional nilpotent algebras is obtained \cite{ack, AKKS, cfk182, ckkk2, hac18, jkk19, kkk}.
Moreover, all central extensions of filiform associative algebras were classified in \cite{kkl18}, central extensions of null-filiform and some filiform Leibniz algebras were
classified in \cite{omirov, is11}, and all central extensions of filiform Zinbiel algebras were classified in \cite{CKKK}.

It should be noted that in \cite{Sund2} by T. Sund the method of central extension is generalized for the solvable Lie algebras.
In  \cite{Sund3} this generalized method of central extensions was applied to get a description of $n$-dimensional solvable Lie algebras with $(n-1)$-dimensional filiform nilradicals over the field of real numbers.
T. Sund called such type of real solvable Lie algebras as filiform solvable Lie algebras and proved that any such $(n+1)$-dimensional filiform solvable Lie algebra is a one-dimensional extension of $n$-dimensional filiform solvable Lie algebra. However, there are not many works in which generalized method of central extension is used for the classification of solvable Lie algebras.

Since the solvable Lie algebras play significant role in the theory of finite-dimensional Lie algebras, there are many techniques to the classification of solvable Lie algebras.
One of the effective methods was introduced by G.M. Mubarakzjanov, which he gave an approach for the investigation of solvable Lie algebras by using their
nilradicals and nil-independent derivations of niradical \cite{Mub}.
Owing to a result of \cite{Mub}, in the papers \cite{AnCaGa3, An2, Ngomo, Snobl, RubWint, Wint}
classification of solvable Lie algebras with the given nilradicals such as abelian, filiform, quasi-filiform and others is obtained.

Filiform Lie algebras are very important subclass in the class of nilpotent Lie algebras and these have the maximal index of nilpotency. Several works are devoted to the classification of filiform Lie algebras \cite{GJK, GK, Ver}.  Natural gradations of nilpotent algebras are very helpful in investigations of properties of those
algebras in general case without restriction on the gradation.
This technique provides a rather deep information on the algebra and
it is more effective when the length of the natural gradation is
sufficiently large. It is well known that up to isomorphism there exist two types of naturally-graded filiform Lie algebras \cite{Ver}.
Solvable Lie algebras with naturally graded filiform nilradicals are classified in \cite{AnCaGa3}, \cite{SnWi}.

The purpose of this article is to find all one-dimensional extension of solvable Lie algebras with naturally graded filiform nilradicals. In order to achieve our goal, we have organized the paper as follows: in Section 2,
we present necessary definitions and results that will be used in the rest of the paper.
We recall the Skjelbred-Sund central extension method for
nilpotent Lie algebras after then we give generalized method of central extension for solvable Lie algebras.
In Section 3, we give all one-dimensional central extensions of naturally graded filiform Lie algebras.
In Section 4, we obtain extension of solvable Lie algebras with naturally graded filiform nilradicals, whose codimension of nilradical is maximal.
In Section 5, we find all one-dimensional extensions of solvable Lie algebras with naturally graded filiform nilradicals with a codimension one.
Finally, in Conclusion Section, we compare the solvable
Lie algebras which are obtained in the Section 4 and Section 5 with the solvable Lie algebras with nilradicals are one-dimensional central extension of naturally graded filiform Lie algebra.
Throughout the paper all the spaces and algebras are assumed finite-dimensional and over the field of complex numbers.

\section{Preliminaries}

In this section we give necessary definitions and preliminary results.


\begin{defn}  An algebra $(L,[-,-])$ over a field $\mathbb{F}$ is called Lie algebra if for any $x,y,z \in L$ the following identities
$$[x,[y,z]]+[y,[z,x]]+[z,[x,y]]=0$$
$$[x,x]=0$$
hold.

\end{defn}

For an arbitrary Lie algebra $L$ we define the \emph{derived} and \emph{central series} as follows:
$$L^{[1]}=L, \ L^{[s+1]}=[L^{[s]},L^{[s]}], \ s\geq 1,$$
$$L^{1}=L, \ L^{k+1}=[L^{k},L], \ k\geq 1.$$

\begin{defn} An $n$-dimensional Lie algebra $L$ is
called \emph{solvable (nilpotent)} if there exists $s\in {N}$ ($k\in {N}$) such that $L^{[s]}=0$ ($L^{k}=0$).
Such minimal number is called \emph{index of solvability} (\emph{nilpotency}).
\end{defn}

The maximal nilpotent ideal of a Lie algebra is said to be the nilradical of the algebra.

\begin{defn}
 An $n$-dimensional Lie algebra $L$ is said to be \emph{filiform} if $dim L^i=n-i$, for $2\leq i\leq n$.
\end{defn}

%

Now let us define a natural gradation for the nilpotent Lie algebras.

\begin{defn} Given a nilpotent Lie algebra $L$ with nilindex $s$, put $L_i=L^i/L^{i+1}, \ 1 \leq i\leq s-1$, and $Gr(L) = L_1 \oplus
L_2\oplus\dots \oplus L_{s-1}$.
Define the product in the vector space $Gr(L)$ as follows:
$$[x + L^{i+1}, y + L^{j+1}]: = [x, y] + L^{i+j+1},$$
where $x \in L^{i} \setminus L^{i+1},$ $y \in L^{j} \setminus L^{j+1}.$ Then $[L_i,L_j]\subseteq L_{i+j}$ and we obtain the graded algebra $Gr(L)$. If $Gr(L)$ and $L$ are isomorphic, then we say that the algebra $L$ is naturally graded.
\end{defn}

It is well known that there are two types of naturally graded filiform Lie algebras. In fact, the second type will appear only in the case when the dimension of the algebra is even.

\begin{thm} \cite{Ver} \label{thm2.8} Any naturally graded filiform Lie algebra is isomorphic to one of the following non isomorphic algebras:
$$\begin{array}{ll}
n_{n,1}:&\left\{ [e_i, e_1]=-[e_1, e_i]=e_{i+1}, \quad 2\leq i \leq n-1.\right.\\[3mm]
Q_{2n}:&\left\{\begin{array}{ll}
[e_i, e_1] =  -[e_1, e_i]=e_{i+1},& 2\leq i \leq 2n-2,\\[2mm]
[e_i, e_{2n+1-i}]  =  -[e_{2n+1-i}, e_i]=(-1)^i e_{2n},& 2\leq i
\leq n.
\end{array}\right.
\end{array}$$
\end{thm}

All solvable Lie algebras whose nilradical is the naturally graded filiform Lie algebra $n_{n,1}$ are classified in \cite{SnWi}. Further  solvable Lie algebras whose nilradical is the naturally graded filiform Lie algebra $Q_{2n}$ are classified in \cite{AnCaGa3}.
It is proved that the dimension of a solvable Lie algebra whose nilradical is isomorphic
to an $n$-dimensional naturally graded filiform Lie algebra is not greater than $n+2$.
Here we give the list of such solvable Lie algebras. We denote by $\mathfrak{s}^{i}_{n,1}$
solvable Lie algebras with nilradical $n_{n,1}$ and codimension one, and by $\mathfrak{s}_{n,2}$ with codimension two.
Similarly, for the solvable Lie algebras with nilradical $Q_{2n}$ we use notations
$\tau_{2n,1}^i$ and $\tau_{2n,2}:$
$$\mathfrak{s}^{1}_{n,1}(\beta)\ : \left\{\begin{array}{llll}
[e_i, e_1] = e_{i+1}, &  2 \leq i \leq n-1, \\[1mm]
[e_1,x]=e_1, \\[1mm]
[e_i, x]  =(i-2+\beta) e_i,  & 2 \leq i \leq n.
\end{array}\right.$$
$$\mathfrak{s}^{2}_{n,1}\ : \left\{\begin{array}{llll}
[e_i, e_1] = e_{i+1}, &  2 \leq i \leq n-1, \\[1mm]
[e_i, x]  = e_i,  & 2 \leq i \leq n.
\end{array}\right.$$
$$\mathfrak{s}^{3}_{n,1}\ : \left\{\begin{array}{lll}
[e_i, e_1] = e_{i+1}, &  2 \leq i \leq n-1, \\[1mm]
[e_1,x]=e_1 +e_2, \\[1mm]
[e_i, x]  = (i-1)e_i,  & 2 \leq i \leq n.
\end{array}\right.$$
\[\mathfrak{s}^{4}_{n,1}(\alpha_3, \alpha_4, \dots, \alpha_{n-1}) :
 \left\{\begin{array}{ll}
[e_i,e_1]=e_{i+1}, & 2 \leq i \leq n-1,\\[1mm]
[e_i,x]=e_i+\sum\limits_{l=i+2}^{n} \alpha_{l+1-i} e_l,  & 2\leq i\leq n.
\end{array}\right.\]
$$\mathfrak{s}_{n,2}:\left\{\begin{array}{lll } [e_i, e_1] = e_{i+1}, &   2 \leq i \leq n-1, \\[1mm]
[e_1, x_1] = e_1, & \\[1mm]
[e_i, x_1]  = (i-2)e_i,  & 3 \leq i \leq n, \\[1mm]
 [e_i, x_2] = e_i, &  2 \leq i \leq n.
\end{array}\right.$$

\[\tau_{2n,1}^1(\alpha):  \left\{\begin{array}{llll}
[e_i,e_1] =e_{i+1},&  2\leq i \leq 2n-2,\\[1mm]
[e_i,e_{2n+1-i}] =(-1)^i e_{2n},& 2\leq i \leq n,\\[1mm]
 [e_1,x]=e_1,&\\[1mm]
 [e_i,x]=(i-2+\alpha) e_i, &2\leq i\leq 2n-1,\\[1mm]
  [e_{2n},x] =(2n-3+2\alpha)e_{2n}.&
\end{array}\right.\]
\[\tau_{2n,1}^2:  \left\{\begin{array}{llll}
[e_i,e_1] =e_{i+1},& 2\leq i \leq 2n-2,\\[1mm]
[e_i,e_{2n+1-i}]=(-1)^i  e_{2n},& 2\leq i \leq n,\\[1mm]
 [e_1,x] =e_1 + e_{2n}, &  \\[1mm]
  [e_i,x]=(i-n) e_i, & 2\leq i\leq 2n-1,\\[1mm]
[e_{2n},x] = e_{2n}.&
\end{array}\right.\]
\[\tau_{2n,1}^3(\alpha_4, \alpha_6, \dots \alpha_{2n-2}):  \left\{\begin{array}{llll}
[e_i,e_1] =e_{i+1}, & 2\leq i \leq 2n-2,\\[1mm]
[e_i,e_{2n+1-i}]  =(-1)^i e_{2n}, & 2\leq i \leq n,\\[1mm]
 [e_{i+2},x] =e_{i+2} + \sum\limits_{k=2}^{\lfloor \frac {2n-3-i} {2} \rfloor}
  \alpha_{2k}\, e_{2k+1+i}, & 0\leq i \leq 2n-3,\\[1mm]
  [e_{2n},x] = 2 e_{2n}.
\end{array}\right. \]
\[\tau_{2n,2}\left\{\begin{array}{llll}
[e_i,e_1]=e_{i+1},& 2\leq i \leq 2n-2,\\[1mm]
[e_i,e_{2n+1-i}]=(-1)^i e_{2n},& 2\leq i \leq n,\\[1mm]
 [e_i,x]=i e_i, & 1\leq i \leq 2n-1,\\[1mm]
[e_{2n},x]=(2n+1) e_{2n}, &\\[1mm]
 [e_i,y]=e_i, & 1\leq i \leq 2n-1,\\[1mm]
 [e_{2n},y]=-[y,e_{2n}]=2 e_{2n}.
\end{array}\right.\]

\subsection{Central extension of nilpotent Lie algebras}
Let $(N, [-,-])$ be a nilpotent Lie algebra over  $\mathbb C$
and $\mathbb V$ be a vector space. The $\mathbb C$-linear space ${\rm Z^{2}}\left(N, \mathbb V \right) $
is defined as the set of all anti-symmetric bilinear maps $\psi \colon {N} \times {N} \longrightarrow {\mathbb V}$ such that
\[ \psi(x,[y,z])+\psi(z,[x,y])+\psi(y,[z,x])=0.\]

These elements will be called 2-{\it cocycles}. For a
linear map $f: N \rightarrow \mathbb V$, if we define $\delta f\colon {N} \times
{N} \longrightarrow {\mathbb V}$ by $\delta f (x,y ) =f([x, y])$, then $\delta f\in {\rm Z^{2}}\left({N},{\mathbb V} \right)$.
We define the set of 2-coboundaries as follows ${\rm B^{2}}\left({N},{\mathbb V}\right) =\left\{\delta f \ | \ f\in {\rm Hom}\left(N,{\mathbb V}\right) \right\} $.
Define the {\it second cohomology space} ${\rm H^{2}}\left(N, {\mathbb V}\right) $ as the quotient space ${\rm Z^{2}}
\left(N,{\mathbb V}\right) \big/{\rm B^{2}}\left( {N},{\mathbb V}\right) $.

For $\theta \in {\rm Z^{2}}\left(
{N},{\mathbb V}\right) $, define on the linear space $\widetilde{N} = {N}\oplus {\mathbb V}$ the
bilinear product $\left[ -,-\right]_{\psi}$ by $$\left[x+u , y+v \right]_{\psi}=
 [x, y] +\psi(x,y)$$ for all $x,y\in {N}, u, v \in {\mathbb V}$.
The algebra $N_{\psi} = (\widetilde{N}, \left[- , - \right]_{\psi})$  is called an $s$-{\it dimensional central extension} of ${N}$ by ${\mathbb V}$. One can easily check that ${N_{\psi}}$ is a Lie
algebra if and only if $\psi \in {\rm Z^2}(N, {\mathbb V})$. Let $\operatorname{Aut}(N)$ be an automorphism group of ${N} $ and let $\varphi \in \operatorname{Aut}({N})$. For $\psi \in
{\rm Z^{2}}\left(N, {\mathbb V}\right) $ define  the action of the group $\operatorname{Aut}(N) $ on ${\rm Z^{2}}\left( N,{\mathbb V}\right) $ by $$\varphi \psi (x,y)
=\psi \left( \varphi \left( x\right) , \varphi \left( y\right) \right) .$$

It is easy to verify that
 ${\rm B^{2}}\left(N,{\mathbb V}\right) $ is invariant under the action of $\operatorname{Aut}(N).$
 So, we have an induced action of  $\operatorname{Aut}(N)$  on ${\rm H^{2}}\left(N,{\mathbb V}\right)$.
Call the
set $\operatorname{Ann}(\psi)=\left\{ x\in N \ | \ \psi  \left( x, N \right) =0\right\} $
the {\it annihilator} of $\psi $ and $Z(N) =\left\{ x\in N \ | \  [x, N] =0\right\}$ the {\it center} of an  algebra $N$. Observe
 that
$Z(N_{\psi}) =(\operatorname{Ann}(\psi) \cap Z(N))
 \oplus {\mathbb V}$.

It is known that any Lie algebra with a non-zero center is a central extension of a lower dimensional algebra.
Now we recall an algorithm of classification all nilpotent Lie algebras of dimension $n$ with the central extension of nilpotent algebras of dimension less than $n.$
In order to solve the isomorphism problem we need to study the
action of $\operatorname{Aut}(N)$ on ${\rm H^{2}}\left( N,{\mathbb V}
\right) $. To do that, let us fix a basis $e_{1},\ldots ,e_{s}$ of ${\mathbb V}$ and $
\psi \in {\rm Z^{2}}\left( N,{\mathbb V}\right) $. Then $\psi $ can be uniquely
written as $\psi \left( x,y\right) =
\displaystyle \sum_{i=1}^{s} \psi_{i}\left( x,y\right) e_{i}$, where $\psi_{i}\in
{\rm Z^{2}}\left( N,\mathbb C\right) $. Moreover, $\operatorname{Ann}(\psi)=\operatorname{Ann}(\psi_{1})\cap\operatorname{Ann}(\psi_{2})\cap\ldots \cap\operatorname{Ann}(\psi_{s})$. Furthermore, $\psi \in
{\rm B^{2}}\left( N,{\mathbb V}\right) $ if and only if all $\psi_{i}\in {\rm B^{2}}\left( N,
\mathbb C\right) $.
It is not difficult to prove that given a Lie algebra $N_{\psi}$ with  $\psi \left( x, y\right) = \displaystyle \sum_{i=1}^{s} \psi_{i}\left(x,y\right) e_{i}\in {\rm Z^{2}}\left( N,{\mathbb V}\right) $ and
$\operatorname{Ann}(\psi)\cap Z( N) =0$, then $N_{\psi}$ has an
annihilator component if and only if $\left[\psi_{1}\right] ,\left[
\psi_{2}\right] ,\ldots ,\left[ \psi_{s}\right] $ are linearly
dependent in ${\rm H^{2}}\left( N,\mathbb C\right) $.

Let ${\mathbb V}$ be a finite-dimensional vector space over $\mathbb C$. The {\it Grassmannian} $G_{k}\left( {\mathbb V}\right) $ is the set of all $k$-dimensional
linear subspaces of ${\mathbb V}$. Let $G_{s}\left( {\rm H^{2}}\left( N,\mathbb C\right) \right) $ be the Grassmannian of subspaces of dimension $s$ in
${\rm H^{2}}\left( N,\mathbb C\right) $. There is a natural action of $\operatorname{Aut}(N)$ on $G_{s}\left( {\rm H^{2}}\left( N,\mathbb C\right) \right) $.
Let $\varphi \in \operatorname{Aut}(N)$. For $W=\left\langle
\left[\psi_{1}\right] ,\left[\psi_{2}\right] ,\dots,\left[ \psi_{s}
\right] \right\rangle \in G_{s}\left( {\rm H^{2}}\left( N,\mathbb C
\right) \right) $ define $\varphi W=\left\langle \left[ \varphi \psi_{1}\right]
,\left[ \varphi \psi_{2}\right] ,\dots,\left[ \varphi \psi_{s}\right]
\right\rangle $.

We denote the orbit of $W\in G_{s}\left(
{\rm H^{2}}\left( N,\mathbb C\right) \right) $ under the action of $\operatorname{Aut}(N)$ by $\operatorname{Orb}(W)$. Given
\[
W_{1}=\left\langle \left[ \psi_{1}\right] ,\left[ \psi_{2}\right] ,\dots,
\left[ \psi_{s}\right] \right\rangle, W_{2}=\left\langle \left[ \vartheta
_{1}\right] ,\left[ \vartheta _{2}\right] ,\dots,\left[ \vartheta _{s}\right]
\right\rangle \in G_{s}\left( {\rm H^{2}}\left( N,\mathbb C\right)
\right),
\]
we easily have that if $W_{1}=W_{2}$, then $ \bigcap\limits_{i=1}^{s}\operatorname{Ann}(\psi_{i})\cap Z\left( N\right) = \bigcap\limits_{i=1}^{s}
\operatorname{Ann}(\vartheta_{i})\cap Z( N) $ and therefore we can introduce
the set
\[
{\bf T}_{s}(N) =\left\{ W=\left\langle \left[ \psi_{1}\right] ,
\left[ \psi_{2}\right], \dots,\left[\psi_{s}\right] \right\rangle \in
G_{s}\left( {\rm H^{2}}\left( N,\mathbb C\right) \right) : \bigcap\limits_{i=1}^{s}\operatorname{Ann}(\psi_{i})\cap Z(N) =0\right\},
\]
which is stable under the action of $\operatorname{Aut}(N)$.

\

Now, let ${\mathbb V}$ be an $s$-dimensional linear space and let us denote by
${\bf E}\left( N,{\mathbb V}\right) $ the set of all {\it non-split $s$-dimensional central extensions} of $N$ by
${\mathbb V}$. By above, we can write
\[
{\bf E}\left( N,{\mathbb V}\right) =\left\{ N_{\psi}:\theta \left( x,y\right) = \sum_{i=1}^{s}\psi_{i}\left(x, y\right) e_{i} \ \ \text{and} \ \ \left\langle \left[\psi_{1}\right] ,\left[\psi_{2}\right] ,\dots,
\left[\psi_{s}\right] \right\rangle \in {\bf T}_{s}(N) \right\} .
\]
We also have the following result.

\begin{lem}
 Let $N_{\psi}, N_{\vartheta}\in {\bf E}\left( N, {\mathbb V}\right) $. Suppose that $\psi \left( x,y\right) =  \displaystyle \sum_{i=1}^{s}
\psi_{i}\left( x,y\right) e_{i}$ and $\vartheta \left( x,y\right) =
\displaystyle \sum_{i=1}^{s} \vartheta _{i}\left( x,y\right) e_{i}$.
Then the Lie algebras $N_{\psi}$ and $N_{\vartheta } $ are isomorphic
if and only if
$$\operatorname{Orb}\left\langle \left[ \theta _{1}\right] ,
\left[ \theta _{2}\right] ,\dots,\left[ \theta _{s}\right] \right\rangle =
\operatorname{Orb}\left\langle \left[ \vartheta _{1}\right] ,\left[ \vartheta
_{2}\right] ,\dots,\left[ \vartheta _{s}\right] \right\rangle .$$
\end{lem}

This shows that there exists a one-to-one correspondence between the set of $\operatorname{Aut}(N)$-orbits on ${\bf T}_{s}\left( N\right) $ and the set of
isomorphism classes of ${\bf E}\left( N,{\mathbb V}\right) $.

\subsection{Extension of solvable Lie algebras}

Now let $L$ be a solvable Lie algebra and $\mathbb V$ be a vector space. Let $\theta : L\rightarrow \operatorname{End} {\mathbb V}$ a representation, $\psi:L\times L\rightarrow {\mathbb V}$ an anti-symmetric bilinear mapping satisfying the condition
\begin{equation}\label{eq2.1}\psi(x,[y,z])+\psi(z,[x,y])+\psi(y,[z,x])+\theta (x)\psi(y,z)+\theta (z)\psi(x,y)+\theta (y)\psi(z,x)=0,\end{equation}
where $x,y,z\in L.$

The bilinear mapping satisfying the previous condition is said to be a 2-cocycle on $L$ with respect to $\theta $.
The set of all such 2-cocycles is denoted by $Z^2(L,\theta, {\mathbb V}).$
The 2-coboundaries on $L$ with respect to $\theta $ are defined as
$$df(x,y)=f([x,y])+\theta\big(\varphi(y)\big)(f(x))-\theta\big(\varphi(x)\big)f(y))$$
for some linear map $f: L \rightarrow {\mathbb V}$ and $\varphi \in \operatorname{Aut}(L)$.
The set of all such 2-coboundaries is denoted by $B^2(L,\theta, {\mathbb V})$ and it is a subset of $Z^2(L,\theta, {\mathbb V}).$
A factor space $Z^2(L,\theta, {\mathbb V})$ by $B^2(L,\theta, {\mathbb V})$ denoted as $H^2(L,\theta, {\mathbb V}),$ i.e., $H^2(L,\theta, {\mathbb V}) = Z^2(L,\theta, {\mathbb V}) / B^2(L,\theta, {\mathbb V}).$

Now we give the extension for the Lie algebra $L$. On the vector space $\tilde{L}=L\oplus {\mathbb V}$
define Lie algebra structure $L(\psi,\theta)$ for the given $\psi\in Z^2(L,\theta, {\mathbb V})$ as follows:
$$[x + u, y + v]=[x,y] +  \psi(x,y) + \theta (x) v-\theta (y) u$$
for all $u,v \in {\mathbb V}, x,y\in L.$ Note that the algebra $L(\psi,\theta)$ is an extension of $L$ by ${\mathbb V}.$

Let us denote by $N$ the nilradical of $L$ and by $Z(N)$ the center of $N$.
 We study Lie algebras   $\tilde{L}=L(\psi,\theta)$ for which the nilradical $\tilde{N}$ is central extension of $N$ by ${\mathbb V}$.
 Denote $$\psi^0=\psi|_{N\times N}, \quad \theta ^0=\theta |_N, \quad \operatorname{Ann}(\psi^0)=\{x\in N \ | \ \psi^0(x,N)=0\}.$$

 \begin{prop}\cite{Sund2} Let
  $\tilde {L}=L(\psi,\theta)$ be an extension of $L$ by ${\mathbb V}$. Then
\begin{itemize}
\item[a)]
 The nilradical $\tilde{N}$ of $\tilde{L}$ is a central extension of $N$ by ${\mathbb V}$  if and only if $\operatorname{Ker} (\theta) \supset N$.

\item[b)] Let $\operatorname{Ker} (\theta) \supset N$. Then $Z(\tilde{N})={\mathbb V}$ if and only if $\operatorname{Ann}(\psi^0)\cap Z(N)=0.$
\end{itemize}
\end{prop}

From this Proposition we easily get that in the case of $L$ is a nilpotent, this extension is central extension  of nilpotent algebras if and only if $\operatorname{Ker}(\theta) =L.$
%

In \cite{Sund2} it is proved that two extensions $L(\psi_1,\theta_1)$ and $L(\psi_2,\theta_2)$ are isomorphic if and only if
the following equations hold
\begin{equation}\label{eq1}\begin{array}{cc}\theta_2\big(\alpha(x)\big) (\beta(a)) = \beta \big(\theta_1(x) (a)\big),\\[1mm]
\psi_2\big(\alpha(x),\alpha(y)\big)= \beta\big( \psi_1(x,y)\big) + f\big([x,y]\big)+  \theta_2\big(\alpha(y)\big)(f(x))-\theta_2\big(\alpha(x)\big)(f( y))
\end{array}\end{equation}
where $\alpha \in \operatorname{Aut}(L), \beta \in Aut (\mathbb V), f \in Hom(L, \mathbb V).$

Using the equation \eqref{eq1} we obtain an action $\operatorname{Aut} L\times \operatorname{Aut} \mathbb V$ to the set $\bigcup\limits_{\theta}Z^2(L,\theta, \mathbb V),$ which
 2-cocycle $\psi$ with respect to $\theta$ acts 2-cocycle $\psi'$ with respect to $\theta'$ as follows:
 $$\theta'(x) (a) = \beta\Big( \theta\big(\alpha(x)\big) \big(\beta^{-1}(a)\big)\Big),$$
$$\psi'(x,y)= \beta\big( \psi(\alpha(x),\alpha(y))\big).$$

In this action we denote by
$\theta'= \beta \circ (\theta \circ \alpha) \circ \beta^{-1}$ and
$\psi' = \beta \circ \psi \circ \alpha. $  We say that $\beta$ is an intertwining operator for representations $\theta'$ and $\theta \circ \alpha.$
It implies from \eqref{eq1} that two extensions $L(\psi_1,\theta_1)$ and $L(\psi_2,\theta_2)$ are isomorphic if and only if there exist $\alpha \in \operatorname{Aut}(L), \beta \in \operatorname{Aut} (\mathbb V)$ such that
$$\psi_2 - \beta \circ \psi_1 \circ \alpha \in B^2(L,\theta_2, \mathbb V)$$
and $\beta$ is an intertwining operator for representations $\theta_2$ and $\theta_1 \circ \alpha.$

Note that the restricted action of $\operatorname{Aut} L\times \operatorname{Aut} \mathbb V$  in $H^2(N, \mathbb V)$ induced an action of $\operatorname{Aut} L$ in the set of all $k$-dimensional subspaces $G_kH^2(N,\mathbb{F})$ of the second cohomology group of $N$ if and only if $L(\psi, \theta)$ contains non abelian factor. We say that an
$\operatorname{Aut} L$ -- orbit  $\Omega$ in $G_kH^2(N,\mathbb{F})$ has no kernel in the center $Z(N)$ if $\operatorname{Ann}(\psi^0)\cap Z(N)=0$ for all $\psi^0 \in \Lambda,$ where $\Lambda$ runs through $\Omega.$

Denote by  $H^2(L, L/N, \mathbb V) $ the space $\bigcup\limits_{\theta}H^2(L,\theta, \mathbb V)$ where $\theta$ runs through those representations of $L$ in $\mathbb V$ which satisfy
$\operatorname{Ker} \theta \supset N$ and $\widetilde{N}/\mathbb V \cong N.$

\begin{thm}\label{thm1}\cite{Sund2} Let $L$ be a solvable Lie algebra over the field $\mathbb{F}$ and $N$ a nilpotent nilradical of $L$. The isomorphism classes of solvable Lie algebras $\tilde{L}$ pos\-ses\-sing nilradical  $\tilde{N}$ with $k$-dimensional center $\mathbb V$, such that $\tilde{L}/ \mathbb V\cong L, \tilde{N}/ \mathbb V\cong N$, and without nonzero abelian direct factors, are in bijective correspondence with those $\operatorname{Aut} L\times \operatorname{Aut} \mathbb V$ -- orbits $\Omega$ in $H^2(L,L/N,\mathbb V) $(under the action $(\alpha, \beta ,\psi)\rightarrow (\beta \circ \psi \circ \alpha$) which satisfy the following conditions:
\begin{itemize}
\item[\emph{1)}]  If $\psi\in \Omega\cap H^2(L,\theta, \mathbb V)$, then $\mathbb V$ can not be written $\mathbb V=\textit{B}\oplus\textit{D}$ where $\textit{B}\supset\psi(L,L), $ $\theta(L)\textit{B}\subset\textit{B}$, and $0\neq \textit{D}\subset C(\theta),$ where $C(\theta)= \{a \in A: \theta(L)a=0\}.$
\item[\emph{2)}]
 $\operatorname{Ann}(\psi^0)\cap Z(N)=0.$
\end{itemize}

\end{thm}

Theorem \ref{thm1} gives an algorithm for constructing all solvable Lie algebras of dimension $n,$ given those algebras of dimension less than $n.$

\section{Central extension of naturally graded filiform Lie algebras}

Let $L$ be a Lie algebra with
a basis $e_{1},e_{2}, \ldots, e_{n}.$ Then by $\Delta_{ij}$ we denote the
bilinear form
$\Delta_{ij}:L\times L \longrightarrow \mathbb C$
with $\Delta_{ij}(e_{l},e_{m}) = \delta_{il}\delta_{jm}.$
The set $\left\{ \Delta_{ij}:1\leq i, j\leq n\right\}$ is a basis for the linear space of
bilinear forms on $L,$ so every $\psi \in
{\rm Z^2}(L,\bf \mathbb V )$ can be uniquely written as $
\psi = \displaystyle \sum_{1\leq i,j\leq n} c_{ij}\Delta _{{i}{j}}$, where $
c_{ij}\in \mathbb C$.

In the following Propositions we give the description of 2-cocycles, 2-coboundaries and second cohomology space of the naturally graded filiform Lie algebras $n_{n,1}$ and $Q_{2n}.$

\begin{prop}\label{prop3.1} For the filiform Lie algebra $n_{n,1}$ the followings are true
\begin{itemize}
\item A basis of $Z^2(n_{n,1}, \mathbb{C}) $ is formed by the following 2-cocycles
$$Z^2(n_{n,1}, \mathbb{C})=\langle\Delta_{i,1}, \ 2\leq i\leq n,  \ \sum\limits_{i=2}^{k}(-1)^i\Delta_{i,2k+1-i},\quad 2\leq k\leq \Big\lfloor\frac{n+1}{2}\Big\rfloor \rangle.$$

%

\item A basis of $B^2(n_{n,1},\mathbb{C})$ is formed by the following 2-coboundaries
 $$B^2(n_{n,1}, \mathbb{C}) =\langle\Delta_{i,1}, \ 2\leq i\leq n-1\rangle.$$

\item A basis of $H^2(n_{n,1},\mathbb{C})$ is formed by the following
$$H^2(n_{n,1},\mathbb{C}) =\Big\langle [\Delta_{n,1}],  \ \Big[\sum\limits_{i=2}^{k}(-1)^i\Delta_{i,2k+1-i}\Big], \ 2\leq k\leq \Big\lfloor\frac{n+1}{2}\Big\rfloor \Big\rangle.$$
\end{itemize}
where $\lfloor n \rfloor$ is an integer part of $n.$
\end{prop}
\begin{proof} The proof follows directly from the definitions of the 2-cocycle and 2-coboundary.
\end{proof}

\begin{prop}\label{prop3.2} For the filiform Lie algebra $Q_{2n},$ the followings are true
\begin{itemize}
\item A basis of $Z^2(Q_{2n}, \mathbb{C})$ is formed by the following 2-cocycles
$$Z^2(Q_{2n}, \mathbb{C})=\langle\Delta_{i,1}, \ 2\leq i\leq 2n-1,  \ \sum\limits_{i=2}^{k}(-1)^i\Delta_{i,2k+1-i},\quad 2\leq k\leq n \rangle.$$
\item A basis of $B^2(Q_{2n},\mathbb{C})$ is formed by the following 2-coboundaries
$$B^2(Q_{2n}, \mathbb{C}) =\langle\Delta_{i,1}, \ 2\leq i\leq 2n-2, \  \sum\limits_{i=2}^{n}(-1)^i\Delta_{i,2n+1-i} \rangle.$$
\item A basis of $H^2(Q_{2n},\mathbb{C})$ is formed by the following cocycles
$$H^2(Q_{2n},\mathbb{C}) =\langle\Delta_{2n-1,1},  \ \sum\limits_{i=2}^{k}(-1)^i\Delta_{i,2k+1-i}, \ 2\leq k\leq n-1 \rangle.$$

\end{itemize}
\end{prop}

\begin{proof} The proof follows directly from the definitions of the 2-cocycle and 2-coboundary.
\end{proof}

From Propositions \ref{prop3.2}, we can easily get that $\operatorname{Ann}(\psi)\cap Z(Q_{2n}) =\{e_n\}$ for any $\psi \in  Z^2(Q_{2n}, \mathbb{C}).$ Thus, there is no non-split central extension of the algebra $Q_{2n}.$ Therefore, we consider central extensions only for the algebra $n_{n,1}.$ For this purpose, first we give automorphism group of the algebra $n_{n,1}.$

\begin{prop}
Let $\varphi\in \operatorname{Aut}(n_{n,1})$. Then
$$\varphi= \left(\begin{array}{cccccccc}
a_1&0&0&0&0&\ldots&0&0\\
a_2&b_2&0&0&0&\ldots&0&0\\
a_3&b_3&a_1b_2&0&0&\ldots&0&0\\
a_4&b_4&a_1b_3&a_1^2b_2&0&\ldots&0&0\\
a_5&b_5&a_1b_4&a_1^2b_3&a_1^3b_2&\ldots&0&0\\
\ldots&\ldots&\ldots&\ldots&\ldots&\ldots&\ldots&\ldots\\
a_{n-1}&b_{n-1}&a_1b_{n-2}&a_1^2b_{n-3}&a_1^3b_{n-4}&\ldots&a_1^{n-3}b_2&0\\
a_n&b_n&a_1b_{n-1}&a_1^2b_{n-2}&a_1^3b_{n-3}&\ldots&a_1^{n-3}b_3&a_1^{n-2}b_2 \end{array}\right)$$
\end{prop}

In the following theorem we give all one-dimensional non-split central extensions of the filiform Lie algebra $n_{n,1}.$

\begin{thm} \label{thm3.4} An arbitrary non-split central extension of the algebra $n_{n,1}$
is isomorphic to one of the following pairwise non-isomorphic algebras
$$n_{n+1,1}, \quad Q_{n+1} \quad \text{and} \quad L_k(2\leq k\leq \Big\lfloor \frac{n}{2}\Big\rfloor):\begin{cases}
[e_i,e_1]=e_{i+1}, & 2\leq i\leq n,\\
[e_i,e_{2k+1-i}]=(-1)^ie_{n+1}, & 2\leq i\leq k,\end{cases}$$

where $\lfloor n \rfloor$ is an integer part of $n.$

\end{thm}
\begin{rem} The algebra $Q_{n+1}$ appears only in case of $n$ is odd.
\end{rem}
\begin{proof}

In the proof of the theorem we consider the cases $n$ is even and odd separately.

\textbf{Case $n$ is even.} Let us denote
$$\nabla_1=[\Delta_{n,1}], \quad \nabla_j=\left[\sum\limits_{i=2}^{j}(-1)^i\Delta_{i,2j+1-i}\right],\quad  2\leq j\leq \frac{n}{2} .$$

Since
$$\left(\begin{array}{cccccccc}
0&0&0&0&0&\ldots&0&-\alpha_1^*\\
0&0&\alpha_2^*&0&\alpha_3^*&\ldots&\alpha_{\frac{n}{2}}^*&0\\
0&-\alpha_2^*&0&-\alpha_3^*&0&\ldots&0&0\\
0&0&\alpha_3^*&0&\alpha_4^*&\ldots&0&0\\
0&-\alpha_3^*&0&-\alpha_4^*&0&\ldots&0&0\\
\ldots&\ldots&\ldots&\ldots&\ldots&\ldots&\ldots\\
0&-\alpha_{\frac{n}{2}}^*&0&0&0&\ldots&0&0\\
\alpha_1^*&0&0&0&0&\ldots&0&0\end{array}\right)=$$
$$(\varphi)^T\cdot\left(\begin{array}{cccccccc}
0&0&0&0&0&\ldots&0&-\alpha_1\\
0&0&\alpha_2&0&\alpha_3&\ldots&\alpha_{\frac{n}{2}}&0\\
0&-\alpha_2&0&-\alpha_3&0&\ldots&0&0\\
0&0&\alpha_3&0&\alpha_4&\ldots&0&0\\
0&-\alpha_3&0&-\alpha_4&0&\ldots&0&0\\
\ldots&\ldots&\ldots&\ldots&\ldots&\ldots&\ldots\\
0&-\alpha_{\frac{n}{2}}&0&0&0&\ldots&0&0\\
\alpha_1&0&0&0&0&\ldots&0&0\end{array}\right)\cdot \varphi$$
for $\varphi \in \operatorname{Aut}(n_{n,1}),$ then for any $\psi=\langle\alpha_1\nabla_1+\alpha_2\nabla_2+\ldots+\alpha_{\frac{n}{2}}\nabla_{\frac{n}{2}}\rangle$, we have the action of the automorphism group
on the subspace $\langle \psi\rangle$ as
$$\langle\alpha_1^*\nabla_1+\alpha_2^*\nabla_2+\ldots+\alpha_{\frac{n}{2}}^*\nabla_{\frac{n}{2}}\rangle$$
where
$$\alpha_1^{*}=\alpha_1a_1^{n-1}b_2,$$
$$\alpha_{k}^{*}=a_1^{2k-3}\left(\sum\limits_{i=2}^{\frac n 2 +2-k}(-1)^i\alpha_{i-2+k}b_i^2+2\sum\limits_{j=k+1}^{\frac{n}{2}}\sum\limits_{i=1}^{j-k}(-1)^{i+1}\alpha_j b_{i+1}b_{2j-i-(2k-3)}\right), \quad 2 \leq k \leq \frac n 2.$$

It is easy to see that $\operatorname{Ann}(\psi)\cap Z(n_{n,1}) =0$ if and only if  $\alpha_1\neq 0.$
Let us consider the following cases:
\begin{itemize}
\item If $\alpha_i=0$ for all $i \ (2\leq\ i \leq \frac{n}{2})$, then
$\alpha_i^{*}=0$  and
we have the representative $\langle\nabla_1\rangle.$

\item If $\alpha_2\neq 0$ and $ \alpha_i=0$ for all $i\ ( 3\leq\ i \leq \frac{n}{2})$, then
$$\begin{cases}\alpha_1^{*}=\alpha_1a_1^{n-1}b_2,\\
\alpha_2^{*}=\alpha_2a_1b_2^2.\end{cases}$$

Choosing $a_1=1, b_2=\frac{\alpha_1}{\alpha_2}$, we have the representative $\langle\nabla_1+\nabla_2\rangle.$

%

\item Now we consider general case. Let $\alpha_k\neq 0$  for some $k$ and $\alpha_i= 0$ for all $i \ (k+1 \leq i \leq \frac{n}{2}).$ Then choosing
$a_1=1,$ $b_2=\frac{\alpha_1}{\alpha_{k}},$ and
 $$b_{2t}=-\frac{1}{2\alpha_{k}b_2}\left(\sum\limits_{i=2}^{t+1}(-1)^i\alpha_{k-t-1+i}b_i^2+2\sum\limits_{j=k-t+2}^{k-1}\sum\limits_{i=2}^{j-k+t}(-1)^{i}\alpha_{j}b_{i}b_{2(j-k+t+1) -i}-2\alpha_{k}\sum\limits_{i=2}^{t-1}(-1)^{i}b_{i+1}b_{2t+1-i}\right),$$
where $2\leq k \leq\frac{n}{2},$ $2 \leq t \leq k-1,$
we have the representative $\langle\nabla_1+\nabla_{k}\rangle.$

It is easy to verify that all previous orbits are different, and we obtain
$${\bf T}_1(n_{n,1})=\operatorname{Orb}\langle\nabla_1\rangle\cup\operatorname{Orb}\langle\nabla_1+\nabla_2\rangle \cup \operatorname{Orb}\langle\nabla_1+\nabla_3\rangle \cup \ldots \cup \operatorname{Orb}\langle\nabla_1+\nabla_{\frac{n}{2}}\rangle.$$
\end{itemize}

Note that for the orbit $\operatorname{Orb}\langle\nabla_1\rangle$ corresponds the algebra $(n+1)$-dimensional naturally graded filiform Lie algebra  $n_{n+1,1}$ and for the orbits
$\operatorname{Orb}\langle\nabla_1+\nabla_k\rangle$ correspond the algebras $L_k$ for $2 \leq k \leq \lfloor \frac n 2 \rfloor.$

\textbf{Case $n$ is odd.}
Let us denote
$$\nabla_1=[\Delta_{n,1}],\quad \nabla_j=\left[\sum\limits_{i=2}^{j}(-1)^i\Delta_{i,2j+1-i}\right],  \quad 2\leq j\leq \frac{n+1}{2}.$$


Then for any
$\psi=\langle\alpha_1\nabla_1+\alpha_2\nabla_2+\ldots+\alpha_{\frac{n+1}{2}}\nabla_{\frac{n+1}{2}}\rangle$, we have the action of the automorphism group
on the subspace $\langle \psi\rangle$ as
$$\langle\alpha_1^*\nabla_1+\alpha_2^*\nabla_2+\ldots+\alpha_{\frac{n+1}{2}}^*\nabla_{\frac{n+1}{2}}\rangle,$$
where
$$\alpha_1^{*}=a_1^{n-2}b_2(\alpha_1 a_1-\alpha_{\frac{n+1}{2}}a_2), \quad \alpha_{\frac{n+1}{2}}^{*}=\alpha_{\frac{n+1}{2}}a_1^{n-2}b_2^2,$$
$$\alpha_{k}^{*}=a_1^{2k-3}\left(\sum\limits_{i=2}^{\frac {n+1} 2 +2-k}(-1)^i\alpha_{i-2+k}b_i^2+2\sum\limits_{j=k+1}^{\frac{n+1}{2}}\sum\limits_{i=1}^{j-k}(-1)^{i+1}\alpha_j b_{i+1}b_{2j-i-(2k-3)}\right), \quad 2 \leq k \leq \frac {n-1} 2,$$

%
%
%

Note that $\operatorname{Ann}(\psi^0)\cap Z(n_{n,1}) =0$ if and only if  $(\alpha_1,\alpha_{\frac{n+1}{2}})\neq (0,0)$.

\begin{itemize}
\item Let $\alpha_{\frac{n+1}{2}}\neq 0.$ Then choosing
 $$a_1=1,\quad a_2=\frac{\alpha_1}{\alpha_{\frac{n+1}{2}}}, \quad b_2=\sqrt{\frac{1}{\alpha_{\frac{n+1}{2}}}},$$
$$b_{2t}=-\frac{1}{2\alpha_{\frac{n+1}{2}}b_2}\left(\sum\limits_{i=2}^{t+1}(-1)^i\alpha_{\frac{n-1}{2}-t+i}b_i^2+2\sum\limits_{j=\frac{n+5}{2}-t}^{\frac{n-1}{2}}
  \sum\limits_{i=2}^{j+t-\frac{n+1}{2}}(-1)^{i}\alpha_{j}b_{i}b_{2(j+t)-n+1-i}-2\alpha_{\frac{n+1}{2}}\sum\limits_{i=2}^{t-1}(-1)^{i}b_{i+1}b_{2t+1-i}\right)$$
  where $2\leq t \leq\frac{n-1}{2}$, we have the representative $\langle\nabla_{\frac{n+1}{2}}\rangle.$

\item Let $\alpha_{\frac{n+1}{2}}= 0,$ then we consider following cases:

\begin{itemize}
\item If $\alpha_1\neq0$ and $\alpha_i=0$ for $2\leq\ i \leq \frac{n-1}{2}$, then by choosing $a_1=1, b_2=\frac{1}{\alpha_1}$, we have the representative $\langle\nabla_1\rangle.$

\item If $\alpha_2\neq 0$ and $\alpha_i=0$ for $3\leq\ i \leq \frac{n-1}{2}$, then we get
$$\begin{cases}\alpha_1^{*}=\alpha_1a_1^{n-1}b_2,\\
\alpha_2^{*}=\alpha_2a_1b_2^2,\end{cases}$$
and choosing $a_1=1, b_2=\frac{\alpha_1}{\alpha_2}$, we have the representative $\langle\nabla_1+\nabla_2\rangle.$

%
%

\item If $\alpha_k\neq 0$  for some $2\leq k \leq\frac{n-1}{2}$ and $\alpha_i= 0$ for $k+1 \leq i \leq \frac{n-1}{2}.$ Then choosing
$$a_1=1, \quad b_2=\frac{\alpha_1}{\alpha_{k}},$$
$$b_{2t}=-\frac{1}{2\alpha_{k}b_2}\left(\sum\limits_{i=2}^{t+1}(-1)^i\alpha_{k-t-1+i}b_i^2+2\sum\limits_{j=k-t+2}^{k-1}\sum\limits_{i=2}^{j+t-k}(-1)^{i}\alpha_{j}b_{i}b_{2(j+t-k+1) -i}-2\alpha_{k}\sum\limits_{i=2}^{t-1}(-1)^{i}b_{i+1}b_{2t+1-i}\right),$$
where $2 \leq t \leq k-1,$
we have the representative $\langle\nabla_1+\nabla_{k}\rangle.$
\end{itemize}

%
%
%

\end{itemize}

It is easy to verify that all previous orbits are different and we obtain
$${\bf T}_1(n_{n,1})=\operatorname{Orb}\langle\nabla_1\rangle \cup \operatorname{Orb}\langle\nabla_1+\nabla_2\rangle \cup \operatorname{Orb}\langle\nabla_1+\nabla_3\rangle \cup \ldots \cup\operatorname{Orb}\langle\nabla_1+\nabla_{\frac{n-1}{2}}\rangle \cup \operatorname{Orb}\langle\nabla_{\frac{n+1}{2}}\rangle$$

Note that for the orbit $\operatorname{Orb}\langle\nabla_{\frac{n+1}{2}}\rangle$ corresponds the algebra $Q_{n+1}$ and for the orbits
$\operatorname{Orb}\langle\nabla_1\rangle$ and $\operatorname{Orb}\langle\nabla_1+\nabla_k\rangle$ for $2 \leq k \leq \frac {n-1} 2$ similarly to the case of $n$ is even corresponds
 the algebras $n_{n+1,1}$ and $L_k $ for $2 \leq k \leq \lfloor \frac n 2 \rfloor.$

\end{proof}

\section{Extension of solvable Lie algebra with filiform nilradicals of codimension $2$}

In this section using an algorithm for constructing solvable Lie algebras which is given in Theorem \ref{thm1}, we obtain all extensions of solvable Lie algebras whose
nilradical is the naturally graded filiform Lie algebras which codimension is nilradical is maximal. Since the nilpotent algebra $Q_{2n}$ has not a non-split central extension, then the solvable algebra with this nilradical also has not a non-split extension. Thus, it is sufficient to consider extension of
solvable Lie algebra $\mathfrak{s}_{n,2}$

\begin{prop} Any automorphism of the algebra $\mathfrak{s}_{n,2}$ has the following form:
$$\begin{cases}\phi(x_1) = x_1 + \beta e_1 + \sum\limits_{k=3}^{n} \Big((k-2)b_{k}+ \beta b_{k-1}\Big)e_k, \\[1mm] \phi(x_2) = x_2 + \sum\limits_{k=2}^{n} b_{k}e_k, \\[1mm]
\phi(e_1) = \alpha_1 e_1 + \alpha_1 \sum\limits_{k=3}^{n} b_{k-1}e_k, \\[1mm]
\phi(e_i) = \alpha_1^{i-2} \alpha_2 e_i + \alpha_1^{i-2} \alpha_2 \sum\limits_{k=3}^{n+2-i} \frac {(-1)^{k}\beta^{k-2}} {(k-2)!}e_{i-2+k}, & 2 \leq i \leq n.
\end{cases}$$
\end{prop}
\begin{proof}The proof follows directly from the definition of an automorphism.
\end{proof}

Now we give the description of $Z^2(\mathfrak{s}_{n,2},\theta, \mathbb{C}),$ i.e., $2$-cocycle on $\mathfrak{s}_{n,2}$ with respect to $\theta$ with one-dimensional abelian algebra $\mathbb{C}=\{e_{n+1}\}.$
Since we consider $\theta: \mathfrak{s}_{n,2} \rightarrow End(\mathbb{C}) $ with condition $\operatorname{Ker} \theta \supset n_{n,1},$ we obtain that $$\theta(x_1)(e_{n+1})=\alpha e_{n+1}, \quad \theta(x_2)(e_{n+1})=\beta e_{n+1}.$$


\begin{prop} \label{DN} Any element $\psi \in Z^2(\mathfrak{s}_{n,2},\theta, \mathbb{C})$ is formed by the following:

1) If $(\alpha, \beta) =(1-n, -1),$ then
$$\begin{array}{llll} \psi(x_1, x_2)= b_{1,2}, & \psi(x_1, e_2)= b_{1,4}, \\
\psi(x_1, e_1)= (n-2)b_{2,3}, & \psi(x_2, e_1)= b_{2,3}, \\
\psi(x_1, e_{j+1})= (j-n)b_{3,j+2}, & \psi(e_1, e_j)= b_{3,j+2}, & 2 \leq j \leq n. \\
\end{array}$$

  2) If  $(\alpha, \beta) =(2-n, -2)$, then
 $$\begin{array}{llll} \psi(x_1, x_2)= b_{1,2},& \psi(x_1, e_1)= \frac{n-3} 2 b_{2,3}, & \psi(x_2, e_{1})= b_{2,3},\\
\psi(x_1, e_2)= (n-2)b_{2,4}, &  \psi(x_2, e_{2})= b_{2,4}, \\
\psi(x_1, e_j)= (n-j)b_{2,j+2}, &  \psi(x_2, e_{j})= b_{2,j+2},& \psi(e_1,e_{j-1})=-b_{2,j+2}&  3 \leq j \leq n, \\
\psi(e_i,e_{n+2-i})=(-1)^ib_{4,n+2},&  2 \leq i \leq \frac{n+1} 2.\\
\end{array}$$
Note that in case of $n$ is even, $b_{4,n+2}=0$.

3) If $(\alpha, \beta) \neq (1-n, -1)$ and $(\alpha, \beta) \neq(2-n, -2)$, then
$$\begin{array}{llll} \psi(x_1, x_2)= b_{1,2} & \psi(x_1, e_1)= (1+\alpha) b_{2,3}, & \psi(x_2, e_{1})= \beta b_{2,3},\\
 \psi(x_1, e_2)= \alpha b_{2,4}, & \psi(x_2, e_{2})= (1+\beta) b_{2,4},\\
\psi(x_1, e_{j})= (j-2+\alpha)b_{2,j+2} & \psi(x_2, e_{j})= (1+\beta)b_{2,j+2}&
\psi(e_1,e_{j-1})=b_{2,j+2} & 3 \leq j \leq n.
\end{array}$$

\end{prop}

\begin{proof} For any $\psi \in Z^2(\mathfrak{s}_{n,2},\theta, \mathbb{C})$
denote by $$\psi(x_1, x_2)= b_{1,2}, \quad \psi(x_1, e_j)= b_{1,j+2}, \quad \psi(x_2, e_j)= b_{2,j+2}, \quad 1\leq j\leq n, $$
$$\psi(e_i, e_j)= b_{i+2,j+2}, \quad 1\leq i, j\leq n.$$

From the condition of a 2-cocycle by straightforward computation we have the following restrictions

 $$\left\{\begin{array}{llll}
b_{i,j}=-b_{i+1,j-1}, & 4\leq i\leq n+1, & i+1\leq j\leq n+2,\\[1mm]
b_{i,n+2}=0,& 5\leq i\leq n-1,\\[1mm]
\beta b_{1,3}=(1+\alpha)b_{2,3},\\[1mm]
(1+\beta )b_{1,j+2}=(j-2+\alpha) b_{2,j+2}, &2\leq j\leq n,\\[1mm]
b_{1,j+3}=((j-1)+\alpha) b_{3,j+2}, & 2\leq j\leq n-1,\\[1mm]
b_{2,j+3}=(1+\beta) b_{3,j+2}, &2\leq j\leq n-1,\\[1mm]
(n-1+\alpha) b_{3,n+2}=0,\\[1mm]
(i+j-6+\alpha) b_{i,j+2}=0, & 4\leq i\leq n+1, & i-1\leq j\leq n,\\[1mm]
(1+\beta) b_{3,n+2}=0,\\[1mm]
(2+\beta) b_{i,j+2}=0,& 4\leq i\leq n+1, & 3\leq j\leq n.
\end{array}\right.$$

Since $b_{i,j}=-b_{i+1,j-1},$ for  $4\leq i\leq n+1, i+1\leq j\leq n+2,$ we have that:
\begin{itemize}
\item If $n$ is even, then
$b_{4,n+2}=-b_{5,n+1}=\ldots=b_{\frac{n}{2}+3,\frac{n}{2}+3}=0,$
\item If $n$ is odd, then
$b_{4,n+2}=-b_{5,n+1}=\ldots=b_{\frac{n+5}{2},\frac{n+7}{2}}=-b_{\frac{n+7}{2},\frac{n+5}{2}}.$
\end{itemize}

Thus, in case of $n$ is even, we have that  $b_{i,n+6-i}=0$ for $4 \leq i \leq \frac n 2 +3.$ Moreover, if
$(\alpha, \beta) =(1-n, -1),$ then we additionally have
$$\begin{cases}
b_{1,3}=(n-2)b_{2,3},\\[1mm]
b_{1,j}=(j-3-n) b_{3,j-1}, & 5\leq j\leq n+2,\\[1mm]
b_{2,j+2}=0, &2\leq j\leq n,\\[1mm]
b_{i,j+2}=0,& 4\leq i\leq n+1, \quad 3\leq j\leq n.
\end{cases}$$

If
$(\alpha, \beta) =(2-n, -2),$ then
$$\begin{cases}
b_{1,3}=\frac {n-3}2 b_{2,3},\\[1mm]
b_{1,j}=(n+2-j)b_{2,j}, & 4\leq j\leq n+2,\\[1mm]
b_{3,j-1}=-b_{2,j}, & 5\leq j\leq n+2,\\[1mm]
b_{3,n+2}=0,\\[1mm]
b_{i,j}=0, & 4\leq i\leq n+1, \quad i+1\leq j\leq n+2, \quad i+j \neq n+6.
\end{cases}$$

Note that in case of $n$ is even, $b_{4,n+2}$ is also equal to zero.

If $(\alpha, \beta) \neq (1-n, -1)$ and $(\alpha, \beta) \neq(2-n, -2)$, then
 $$\begin{cases}
\beta b_{1,3}=(1+\alpha)b_{2,3},\\[1mm]
(1+\beta )b_{1,j}=(j-4+\alpha) b_{2,j} , &4\leq j\leq n+2,\\[1mm]
b_{2,j}=(1+\beta) b_{3,j-1}, & 5\leq j\leq n+2,\\[1mm]
b_{3,n+2}=0,\\[1mm]
b_{i,j+2}=0, & 4\leq i\leq n+1,  \quad i-1\leq j\leq n.
\end{cases}$$

\end{proof}

Now we determine the elements of the space $B^2(\mathfrak{s}_{n,2},\theta, \mathbb{C}).$ Putting
$$f(x_1)=c_1e_{n+1}, \quad f(x_2)=c_2e_{n+1}, \quad f(e_i)=c_{i+2}e_{n+1}, \quad 3\leq i\leq n,$$
for any automorphism $\phi \in \operatorname{Aut}(\mathfrak{s}_{n,2})$ considering
$$df(y,z)=f\big([y,z]\big)+\theta\big(\phi(z)\big)(f(y))-\theta\big(\phi(y)\big)(f(z))$$ we have

$$\begin{array}{llll} df(x_1,x_2)=\beta c_1-\alpha c_2, \\ df(x_1,e_1)=-(1+\alpha)c_3,\\
df(x_1,e_2)=-\alpha c_4, \\ df(x_1,e_j)=-(j-2+\alpha)c_{j+2}, & 3\leq j\leq n, \\ df(x_2,e_1)=-\beta c_3, \\
df(x_2,e_j)=-(1+\beta)c_{j+2}, & 2\leq j\leq n, \\
df(e_1,e_j)=-c_{j+3}, & 2\leq j\leq n-1. \end{array}$$

It is not difficult to see that in cases of
\begin{itemize}\item  $(\alpha, \beta) \neq (1-n, -1)$ and $(\alpha, \beta) \neq(2-n, -2),$
\item $(\alpha, \beta) =(2-n, -2)$ and $n$ is even,
 \end{itemize}
 we have
 $\operatorname{Ann}(\psi^0)\cap Z(n_{n,1})=\{e_n\} \neq 0.$ Therefore, to get a non-split extension of the solvable Lie algebra $\mathfrak{s}_{n,2}$ it is enough to consider the cases
 $(\alpha, \beta) =(1-n, -1)$ and $(\alpha, \beta) =(2-n, -2),$ $n$ is odd.

In this two cases we have $\dim Z^2(\mathfrak{s}_{n,2},\theta, \mathbb{C})=n+2,$ $\dim B^2(\mathfrak{s}_{n,2},\theta, \mathbb{C})=n+1$
which implies $\dim H^2(\mathfrak{s}_{n,2},\theta, \mathbb{C}) =1$
 and a basis of $H^2(\mathfrak{s}_{n,2},\theta, \mathbb{C})$ is formed by the following cocycle
 $$H^2(\mathfrak{s}_{n,2},\theta,\mathbb{C}) =\langle [\psi]\rangle, \quad \psi(e_n, e_1) = e_{n+1}, \quad \ (\alpha, \beta) =(1-n, -1),$$
 $$H^2(\mathfrak{s}_{n,2},\theta, \mathbb{C}) =\langle[\psi]\rangle, \quad \psi(e_{n+2-i},e_i) = (-1)^ie_{n+1}, \ 2 \leq i \leq \frac{n+1} 2, \quad (\alpha, \beta) =(2-n, -2).$$

Now define new products of the extension algebra $\widetilde{L} = \mathfrak{s}_{n,2} \oplus \{e_{n+1}\}.$ In case of $(\alpha, \beta) =(1-n, -1)$ we have
$$\begin{array}{l}[e_n, e_1]=\psi(e_n,e_1) =e_{n+1},\\[1mm]
[e_{n+1}, x_1]=-\theta(x_1)e_{n+1}=(n-1)e_{n+1},\\[1mm]
[e_{n+1}, x_2]=-\theta(x_2)e_{n+1} = e_{n+1}.\end{array}$$

In the case of $(\alpha, \beta) =(2-n, -2)$ and $n$ is odd we have the following new products
$$\begin{array}{lll}[e_{n+2-i}, e_i]=\psi(e_{n+2-i}, e_i) = (-1)^ie_{n+1}, & 2\leq i \leq \frac{n+1} 2,\\[1mm]
[e_{n+1}, x_1]=-\theta(x_1)e_{n+1}=(n-2)e_{n+1},\\[1mm]
[e_{n+1}, x_2]=-\theta(x_2)e_{n+1}=2e_{n+1}.\end{array}$$

Therefore, we get the following main result of this Section.

\begin{thm} \label{thm4.3} Let $\widetilde{L}$ be an extension of the solvable Lie algebra $\mathfrak{s}_{n,2},$  then $dim (\widetilde{L}) = dim(\mathfrak{s}_{n,2})+1$ and $\widetilde{L}$ is isomorphic to one of the algebras $\mathfrak{s}_{n+1,2}$ and $\tau_{n+1,2}.$
%
\end{thm}

Note that the algebra $\tau_{n+1,2}$ appears only in case of $n$ is odd, i.e., in case of $n$ is even there exists only one extension $\mathfrak{s}_{n+1,2}.$

\section{Extension of solvable Lie algebra with filiform nilradicals of codimension $1$}

In this section we obtain all one-dimensional extensions of solvable Lie algebras $$\mathfrak{s}^{1}_{n,1}(\beta), \quad \mathfrak{s}^{2}_{n,1}, \quad \mathfrak{s}^{3}_{n,1}, \quad \mathfrak{s}^{4}_{n,1}(\alpha_3, \alpha_4, \dots, \alpha_{n-1}).$$

First, we give the description of the group of automorphisms of these algebras.

\begin{prop}
Any automorphism of the  algebra $\mathfrak{s}^{1}_{n,1}(\beta)$ has the following form:
$$\phi(x) = x +a_1e_1+\frac{1}{b_1}\sum\limits_{k=2}^{n-1}(\beta+k-2) b_{k+1}e_{k}+a_ne_n,
\quad \phi(e_1) = b_1e_1+\sum\limits_{k=3}^{n} b_{k}e_k, $$
 $$\phi(e_2) = c_2\sum\limits_{k=2}^{n} \frac{(-1)^k}{(k-2)!}a_1^{k-2}e_k,\quad \phi(e_i) =b_{1}^{i-2}c_2\sum\limits_{k=i}^{n} \frac{(-1)^{k-i}}{(k-i)!}a_1^{k-i}e_k, \quad 3 \leq i \leq n.$$
 Any automorphism of the algebra $\mathfrak{s}^{2}_{n,1}$ has the following form:
$$\phi(x) = x +  \sum\limits_{k=2}^{n} a_{k}e_{k},
\quad \phi(e_1) = b_1e_1+b_1\sum\limits_{k=3}^{n} a_{k-1}e_k, $$
 $$\phi(e_2) = \sum\limits_{k=2}^{n} c_{k}e_k,\quad \phi(e_i) =b_{1}^{i-2}\sum\limits_{k=i}^{n} c_{k-i+2}e_k, \quad 3 \leq i \leq n.$$

Any automorphism of the algebra $\mathfrak{s}^{3}_{n,1}$ has the following form:
$$\phi(x) = x +  \sum\limits_{k=1}^{n} a_{k}e_{k},
\quad \phi(e_1) = \sum\limits_{k=1}^{n} b_{k}e_k,$$
 $$\phi(e_i) =b_{1}^{i-1}\sum\limits_{k=i}^{n} \frac {(-1)^{k+2-i}a_{1}^{k-i}} {(k-i)!}e_{k}, \quad 2 \leq i \leq n.$$

Any automorphism of the algebra $\mathfrak{s}^{4}_{n,1}(\alpha_3, \alpha_4, \dots, \alpha_{n-1})$ has the following form:
$$\phi(x) = x +  \sum\limits_{k=2}^{n-1}\Bigl(b_{k+1}+\sum\limits_{l=3}^{k-1} \alpha_{l}b_{k+2-l}\Bigr)e_k+a_{n}e_{n},$$
$$ \phi(e_1) = e_1+\sum\limits_{k=3}^{n} b_{k}e_{k},
\quad \phi(e_i) =\sum\limits_{k=i}^{n} c_{k+2-i}e_{k}, \quad 2 \leq i \leq n.$$
\end{prop}
\begin{proof}The proof follows directly from the definition of an automorphism.
\end{proof}

Now we give the description of $2$-cocycles with respect to $\theta$ of these solvable Lie algebras with one-dimensional abelian algebra $\mathbb{C}=\{e_{n+1}\}.$
Note that a basis of these algebras is $\{x, e_1, e_2, \dots, e_n\}$ and the nilradical is $n_{n,1} = \{e_1, e_2, \dots, e_n\}.$
Thus, we have that $$\theta(x)(e_{n+1})=\gamma e_{n+1}, \quad
\theta(e_i)(e_{n+1})=0, \quad 1 \leq i \leq n.$$

\begin{prop} \label{DN5.2} Any element $\psi \in Z^2(\mathfrak{s}^{1}_{n,1}(\beta),\theta, \mathbb{C})$ is formed by the following:

1) If $\gamma=1-n-\beta$ and $\beta=n+2-2k$ for some
$k \ (2\leq k \leq \Big \lfloor\frac{n+1} 2\Big\rfloor),$ then
$$\begin{array}{llll} \psi(x, e_1)= b_{1,2}, & \psi(x, e_2)= b_{1,3}, \\
\psi (e_1,e_{j-1})=b_{2,j}, & \psi(x, e_{j})=(j-1-n) b_{2,j}, & 3 \leq j \leq n, \\
\psi (e_1,e_{n})=b_{2,n+1}, & \psi (e_i,e_{2k+1-i})=(-1)^ib_{3,2k}, & 2 \leq i \leq k.\\
\end{array}$$


2) If $\gamma=1-n-\beta$ and $\beta\neq n+2-2k$ for any
$k\ (2\leq k \leq \Big \lfloor\frac{n+1} 2\Big\rfloor),$ then
$$\begin{array}{llll} \psi(x, e_1)= b_{1,2}, & \psi(x, e_2)= b_{1,3}, \\
\psi (e_1,e_{j-1})=b_{2,j}, & \psi(x, e_{j})=(j-1-n) b_{2,j}, & 3 \leq j \leq n,\\
\psi (e_1,e_{n})=b_{2,n+1}. &
\end{array}$$

3) If  $\gamma=2-n-2\beta$ and $\beta \neq 1$, then
 $$\begin{array}{llll}  \psi(x, e_1)= b_{1,2}, & \psi(x, e_2)= b_{1,3},\\
\psi(e_1, e_{j-1})= b_{2,j}, &  \psi(x, e_{j})=(j-n-\beta) b_{2,j},&  3 \leq j \leq n, \\
\psi(e_i,e_{n+2-i})=(-1)^ib_{3,n+1},&2 \leq i \leq\Big \lfloor\frac{n+1} 2\Big\rfloor.\\
\end{array}$$

4) If $\gamma\neq 1-n-\beta,$ $\gamma \neq2-n-2\beta$ and  $\beta=\frac {3-2k-\gamma} 2$ for some $k \ (2\leq k \leq \Big \lfloor \frac{n} 2 \Big\rfloor),$ then
$$\begin{array}{llll} \psi(x, e_1)= b_{1,2}, & \psi(x, e_2)= b_{1,3},\\
\psi(e_1, e_{j-1})= b_{2,j}, &  \psi(x, e_{j})=(j-3+\beta+\gamma) b_{2,j},&  3 \leq j \leq n,\\
\psi (e_i,e_{2k+1-i})=(-1)^ib_{3,2k}, & & 2 \leq i \leq k.
\end{array}$$

5) If $\gamma\neq 1-n-\beta,$ $\gamma \neq2-n-2\beta$ and $\beta\neq\frac {3-2k-\gamma} 2$ for any $k \ (2\leq k \leq  \Big \lfloor \frac{n} 2 \Big\rfloor),$ then
$$\begin{array}{llll} \psi(x, e_1)= b_{1,2}, & \psi(x, e_2)= b_{1,3},\\
\psi(e_1, e_{j-1})= b_{2,j}, &  \psi(x, e_{j})=(j-3+\beta+\gamma) b_{2,j},&  3 \leq j \leq n.
\end{array}$$

Note that in case of $n$ is even $b_{3,n+1}=0$.

\end{prop}

\begin{proof} For any $\psi \in Z^2(\mathfrak{s}^{1}_{n,1}(\beta),\theta, \mathbb{C})$
denote by $$\psi(x, e_j)= b_{1,j+1},  \quad 1\leq j\leq n, \quad \psi(e_i, e_j)= b_{i+1,j+1}, \quad 1\leq i, j\leq n.$$

From the condition of a 2-cocycle by straightforward computation we have the following restrictions

$$\begin{cases}
b_{i,j}=-b_{i+1,j-1}, & 3\leq i\leq n-2,\quad i+3\leq j\leq n+1,\\
b_{1,j}=(j-3+\beta+\gamma)b_{2,j-1},& 4\leq j\leq n+1,\\
(i+j-6+2\beta+\gamma) b_{i,j}=0,& 3\leq i \leq n,\quad i+1\leq j \leq n+1,\\
b_{i,i+2}=0, & 3\leq i\leq n-1,\\
b_{i,n+1}=0,& 4\leq i\leq n,\\
(n-1+\beta+\gamma)b_{2,n+1}=0.
\end{cases}$$

Since $b_{i,j}=-b_{i+1,j-1}$ for $3\leq i\leq n-2,$ $i+3\leq j\leq n+1,$ we have that

\begin{itemize}
\item If $n$ is even, then
$b_{3,n+1}=-b_{4,n}=\ldots=b_{\frac{n}{2}+2,\frac{n}{2}+2}=0,$
\item If $n$ is odd, then
$b_{3,n+1}=-b_{4,n}=\ldots=b_{\frac{n+5}{2},\frac{n+3}{2}}=-b_{\frac{n+3}{2},\frac{n+5}{2}}.$
\end{itemize}

Thus, in the case of $n$ is even, we have that  $b_{i,n+4-i}=0$ for $3 \leq i \leq \frac n 2 +2.$ Moreover,
we additionally have following subcases:
\begin{enumerate}
\item If $\gamma=1-n-\beta$ and $\beta=n+2-2k$ for some
$k \ (2\leq k \leq \Big \lfloor\frac{n+1} 2\Big\rfloor),$  then we get
$$\begin{cases}
b_{1,i}=(i-2-n)b_{2,i-1},& 4\leq i \leq n+1,\\[1mm]
b_{i,j}=0,& 3\leq i\leq n, \quad i+1\leq j\leq n+1, \quad i+j \neq 2k+3,\\
b_{i+1,2k+2-i}=(-1)^ib_{3,2k}, & 2 \leq i \leq k.
\end{cases}$$

\item If $\gamma=1-n-\beta$ and $\beta\neq n+2-2k$ for any
$k\ (2\leq k \leq \Big \lfloor\frac{n+1} 2\Big\rfloor),$  then we have
$$\begin{cases}
b_{1,i}=(i-2-n)b_{2,i-1},& 4\leq i \leq n+1,\\[1mm]
b_{i,j}=0,& 3\leq i\leq n, \quad i+1\leq j\leq n+1.
\end{cases}$$

\item If
$\gamma=2-n-2\beta$ and $\beta \neq 1$, then
$$\begin{cases}
b_{1,i}=(i-1-n-\beta)b_{2,i-1}, & 4\leq j\leq n+1,\\[1mm]
b_{2,n+1}=0,\\[1mm]
b_{i,j}=0, & 3\leq i\leq n, \quad i+1\leq j\leq n+1, \quad i+j \neq n+4.
\end{cases}$$


\item If $\gamma \neq 1-n-\beta,$ $\gamma \neq 2-n-2\beta$ and
 $\beta=\frac {3-2k-\gamma} 2$ for some $k \ (2\leq k \leq \Big \lfloor \frac{n} 2 \Big\rfloor),$ then we have
 $$\begin{cases}
b_{1,j}=(i-3+\beta+\gamma) b_{2,j-1} , & 4\leq j\leq n+1,\\[1mm]
b_{2,n+1}=0,\\[1mm]
b_{i,j}=0, & 3\leq i\leq n,  \quad i+1\leq j\leq n+1, \quad i+j \neq 2k+3,\\[1mm]
b_{i+1,2k+2-i}=(-1)^ib_{3,2k}, & 2 \leq i \leq k.
\end{cases}$$

\item If $\gamma \neq 1-n-\beta,$ $\gamma \neq 2-n-2\beta$ and
 $\beta\neq \frac {3-2k-\gamma} 2$ for any $k \ (2\leq k \leq \Big \lfloor \frac{n} 2 \Big\rfloor),$ then we have
 $$\begin{cases}
b_{1,j}=(i-3+\beta+\gamma) b_{2,j-1} , & 4\leq j\leq n+1,\\[1mm]
b_{2,n+1}=0,\\[1mm]
b_{i,j}=0, & 3\leq i\leq n,  \quad i+1\leq j\leq n+1.
\end{cases}$$
\end{enumerate}

\end{proof}

Now we determine the space of 2-coboundaries with respect to $\theta$. Putting
$$f(x)=c_0e_{n+1},  \quad f(e_i)=c_{i}e_{n+1}, \ 1\leq i\leq n$$
for any automorphism $\phi \in \operatorname{Aut}(\mathfrak{s}^{1}_{n,1}(\beta))$ considering
$$df(y,z)=f\big([y,z]\big)+\theta\big(\phi(z)\big)(f(y))-\theta\big(\phi(y)\big)(f(z))$$ we have
$$\begin{array}{llll} df(x,e_1)=-(1+\gamma) c_1, \\
df(x,e_i)=-(i-2+\beta+\gamma)c_{i}, & 2\leq i\leq n, \\
df(e_1,e_i)=-c_{i+1}, & 2\leq i\leq n-1. \end{array}$$

Thus, we get
$$\dim B^2(\mathfrak{s}^{1}_{n,1}(\beta),\theta,\mathbb{C})  = \left\{\begin{array}{llll} n & \text{if} & \gamma \neq -1,\\
n-1 & \text{if} & \gamma = -1. \end{array}\right.$$

It is not difficult to see that to get a non-split extension of the solvable Lie algebra $\mathfrak{s}^{1}_{n,1}(\beta)$ it is enough to consider the cases
 $\gamma=1-n-\beta$ and $\gamma=2-n-2\beta,$ $n$ is odd.
In these cases we have the followings


\begin{enumerate}
\item If $\gamma =1-n-\beta$ and $\beta=n+2-2k$ for some $k \ (2\leq k\leq \Big\lfloor\frac{n+1}{2}\Big\rfloor),$ then
 $$H^2(\mathfrak{s}^{1}_{n,1}(\beta),\theta,\mathbb{C}) =\Big\langle  [\Delta_{n+1,2}], \Big[\sum_{i=2}^{k}(-1)^{i}\Delta_{i+1,2k+2-i}\Big] \Big\rangle.$$

\item If $\gamma =1-n-\beta$ and $\beta \neq n+2-2k$ for any  $k \ (2\leq k\leq \Big\lfloor\frac{n+1}{2}\Big\rfloor),$ then
 $$H^2(\mathfrak{s}^{1}_{n,1}(\beta),\theta,\mathbb{C}) =\Big\langle  [\Delta_{n+1,2}]\Big\rangle.$$

 \item If $\gamma=2-n-2\beta,$ $\beta \neq -1,$ $\gamma \neq - 1,$ $n$ is odd, then
  $$H^2(\mathfrak{s}^{1}_{n,1}(\beta),\theta, \mathbb{C}) =\Big\langle  \Big[\sum_{i=3}^{\lfloor\frac{n+1}{2}\rfloor}(-1)^{i}\Delta_{i,n+4-i}\Big]\Big\rangle .$$

 \item If $\gamma =- 1,$ $\beta = \frac {3-n} 2,$ $n$ is odd, then
  $$H^2(\mathfrak{s}^{1}_{n,1}(\beta),\theta, \mathbb{C}) =\Big\langle [\Delta_{1,2}],  \Big[\sum_{i=3}^{\lfloor\frac{n+1}{2}\rfloor}(-1)^{i}\Delta_{i,n+4-i}\Big]\Big\rangle .$$

 \end{enumerate}

Here for bilinear form $\Delta_{i,j}$ we use the denotation similarly in Section 3, respect to the basis  $\{x, e_{1},e_{2}, \ldots, e_{n}\}.$

\begin{thm} \label{thm5.9} Let $\widetilde{L}$ be a one-dimensional extension of the solvable Lie algebra $\mathfrak{s}^{1}_{n,1}(\beta),$  then $\widetilde{L}$ is isomorphic to one of the following algebras
$$\mathfrak{s}^{1}_{n+1,1}(\beta), \quad \tau_{n+1,1}^1(\beta), \quad \tau_{n+1,1}^2 $$
and
$$ \widetilde{L_k}(2\leq k\leq \Big\lfloor\frac{n}{2}\Big\rfloor):\begin{cases}
[e_i, e_1] = e_{i+1}, &   2 \leq i \leq n, \\
[e_i,e_{2k+1-i}]=(-1)^ie_{n+1},& 2\leq i\leq k,\\
[e_1, x]=e_1, \\
[e_i, x]  = (n+i-2k)e_k,  & 2 \leq i \leq n+1.
\end{cases}$$

%
%
%
%


\end{thm}

\begin{proof}

(1) Let $\gamma=1-n-\beta$ and $\beta=n+2-2k$ for some  $2\leq k\leq \Big\lfloor\frac{n+1}{2}\Big\rfloor.$ Denote by
$$ \nabla_1=[\Delta_{n+1,2}], \quad \nabla_k=\Bigl[{\sum_{i=2}^{k}(-1)^{i}\Delta_{i+1,2k+2-i}}\Big]. $$

Then any $\psi=\langle \delta_1\nabla_1+\delta_k\nabla_k\rangle,$ acts on the subspace $\langle \delta_1^*\nabla_1+\delta_k^*\nabla_k\rangle$ by action of the automorphism group as
$$\begin{array}{llll} \delta_1^*=\delta_1c_2b_{1}^{n-1}, & \delta_k^*=\delta_k c_2^2b_1^{2k-3} & \text{if} & k < \frac {n+1} 2, \\[1mm] \delta_1^*=c_2b_1^{n-2}(\delta_1b_1+\delta_{\frac {n+1} 2}b_{\frac {n+1} 2}), & \delta_k^*=\delta_k c_2^2b_1^{n-2} & \text{if} & k = \frac {n+1} 2. \end{array}$$

Note that in case of $k < \frac {n+1} 2$ we have $\operatorname{Ann}(\psi^0)\cap Z(n_{n,1}) =0$ if and only if  $\delta_1\neq 0.$ Thus, we have
\begin{itemize}
\item If $\delta_k=0$, then
$\delta_k^{*}=0$  and
we have the representative $\langle\nabla_1\rangle$ and obtain the algebra $\mathfrak{s}^{1}_{n+1,1}(\beta)$ for $\beta=n+2-2k.$
\item If $\delta_k\neq 0$, then choosing by $b_1=1, c_2=\frac{\delta_1}{\delta_k}$, we have the representative $\langle\nabla_1+\nabla_k\rangle$ obtain the algebra
    $\widetilde{L_k},$ $2\leq k \leq  \Big\lfloor\frac{n}{2}\Big\rfloor.$
\end{itemize}

In case of $k = \frac {n+1} 2$ (where $n$ is odd)
we have that $\beta =1$ and
\begin{itemize}
\item If $\delta_{\frac {n+1} 2}=0$, then
$\delta_{\frac {n+1} 2}^{*}=0$  and
we have the representative $\langle\nabla_1\rangle$ and obtain the algebra $\mathfrak{s}^{1}_{n+1,1}(\beta)$ for $\beta=1.$
\item If $\delta_{\frac {n+1} 2}\neq 0$, then
choosing by $b_{\frac {n+1} 2} = - \frac{\delta_1b_1} {\delta_{\frac {n+1} 2}},$
we have the representative $\langle\nabla_{\frac {n+1} 2}\rangle$ and obtain the algebra $\tau_{n+1,1}^1(\beta)$ for $ \beta =1.$
\end{itemize}

(2) Let  $\gamma=1-n-\beta$ and $\beta \neq n+2-2k$ for any  $2\leq k\leq \Big\lfloor\frac{n+1}{2}\Big\rfloor.$
Then we have only one orbit $\langle [\Delta_{n+1,2}]\rangle$ and obtain the algebra $\mathfrak{s}^{1}_{n+1,1}(\beta)$ for $\beta \neq n+2-2k .$

(3) Let $\gamma=2-n-2\beta$, $\beta \neq 1$ and $\gamma \neq -1$. Then we have the orbit
$$\Big\langle \Big[ \sum_{i=3}^{\lfloor\frac{n+1}{2}\rfloor}(-1)^{i}\Delta_{i,n+4-i}\Big] \Big\rangle$$
and obtain the algebra $\tau_{n+1,1}^1(\beta)$ for $\beta \neq 1.$

(4) Let $\gamma =- 1,$ $\beta = \frac {3-n} 2$. Then denote by
$$ \nabla_1=[\Delta_{2,1}], \quad \nabla_2=\Big[{\sum_{i=3}^{\lfloor\frac{n+1}{2}\rfloor}(-1)^{i}\Delta_{i,n+4-i}}\Big]. $$
 Then any $\psi=\langle \delta_1\nabla_1+\delta_2\nabla_2\rangle,$ acts on the subspace $\langle \delta_1^*\nabla_1+\delta_2^*\nabla_2\rangle$ by action of the automorphism group as
$$\delta_1^*=\delta_1b_{1}, \quad  \delta_2^*=\delta_2 c_2^2b_1^{n-1}.$$

If $\delta_1 =0,$ then we have the orbit $\langle\nabla_2 \rangle$ and obtain the algebra  $\tau_{n+1,1}^1(\beta)$ for $\beta = \frac {3-n} 2.$

If $\delta_1 \neq 0,$ then we have the orbit $\langle\nabla_1 + \nabla_2 \rangle$ and obtain the algebra  $\tau_{n+1,1}^2.$

\end{proof}

Similar to the above, we give the description of $Z^2(\mathfrak{s}^{2}_{n,1},\theta, \mathbb{C}),$ i.e., $2$-cocycles  with respect to $\theta$ with one-dimensional abelian algebra $\mathbb{C}=\{e_{n+1}\}.$

\begin{prop} \label{DN5.4} Any element $\psi \in Z^2(\mathfrak{s}^{2}_{n,1},\theta, \mathbb{C})$ is formed by the following:

1) If $\gamma=-1,$ then
$$\begin{array}{llll} \psi(x, e_1)= b_{1,2}, & \psi(x, e_2)= b_{1,3}, &
\psi(e_1, e_j)= b_{2,j+1}, & 2 \leq j \leq n. \\
\end{array}$$

  2) If  $\gamma=-2$, then
 $$\begin{array}{llll}  \psi(x, e_1)= b_{1,2}, & \psi(x, e_2)= b_{1,3},\\
\psi(e_1, e_{j-1})= b_{2,j}, &  \psi(x, e_{j})=- b_{2,j},&  3 \leq j \leq n, \\
\psi(e_i,e_{2k+1-i})=(-1)^ib_{3,2k},& 2\leq k\leq \Big \lfloor\frac{n+1} 2\Big\rfloor, &  2 \leq i \leq k.\\
\end{array}$$
Note that in case of $n$ is even $b_{3,n+1}=0$.

3) If $\gamma\neq -1$ and $\gamma \neq-2$, then
$$\begin{array}{llll} \psi(x, e_1)= b_{1,2}, & \psi(x, e_2)= b_{1,3},\\
\psi(e_1, e_{j-1})= b_{2,j}, &  \psi(x, e_{j})=(1+\gamma) b_{2,j},&  3 \leq j \leq n.\\
\end{array}$$
\end{prop}

\begin{proof} The proof is similar to the Proposition \ref{DN5.2}.

\end{proof}

Now we determine the space of 2-coboundaries with respect to $\theta$. Putting
$$f(x)=c_0e_{n+1},  \quad f(e_i)=c_{i}e_{n+1}, \ 1\leq i\leq n$$
for any automorphism $\phi \in \operatorname{Aut}(\mathfrak{s}^{2}_{n,1})$ considering
$$df(y,z)=f\big([y,z]\big)+\theta\big(\phi(z)\big)(f(y))-\theta\big(\phi(y)\big)(f(z))$$ we have
$$\begin{array}{llll} df(x,e_1)=-\gamma c_1, \\
df(x,e_i)=-(1+\gamma)c_{i}, & 2\leq i\leq n, \\
df(e_1,e_i)=-c_{i+1}, & 2\leq i\leq n-1. \end{array}$$

It is not difficult to see that in cases of
 $\gamma \neq -1,$ $\gamma \neq-2,$ and $\gamma =-2,$ $n$ is even,
  we have
 $\operatorname{Ann}(\psi^0)\cap Z(n_{n,1})=\{e_n\} \neq 0.$ Therefore, to get a non-split extension of the solvable Lie algebra $\mathfrak{s}^{2}_{n,1}$ it is enough to consider the cases
 $\gamma=-1$ and $\gamma=-2,$ $n$ is odd.

In the first case we have $\dim Z^2(\mathfrak{s}^{2}_{n,1},\theta, \mathbb{C})=n+1,$ $\dim B^2(\mathfrak{s}^{2}_{n,1},\theta, \mathbb{C})=n-1$
which implies $\dim H^2(\mathfrak{s}^{2}_{n,1},\theta, \mathbb{C}) =2$
 and a basis of $H^2(\mathfrak{s}^{2}_{n,1},\theta, \mathbb{C})$ is formed by the following cocycles
 $$H^2(\mathfrak{s}^{2}_{n,1},\theta,\mathbb{C}) =\langle [\Delta_{3,1}], [\Delta_{n+1,2}] \rangle, \quad \gamma=-1.$$

 In  second case we have $\dim Z^2(\mathfrak{s}^{2}_{n,1},\theta, \mathbb{C})=\frac{3n-1}{2},$ $\dim B^2(\mathfrak{s}^{2}_{n,1},\theta, \mathbb{C})=n$
which implies $\dim H^2(\mathfrak{s}^{2}_{n,1},\theta, \mathbb{C}) =\frac{n-1}{2}$
 and a basis of $H^2(\mathfrak{s}^{2}_{n,1},\theta, \mathbb{C})$ is formed by the following cocycles

 $$H^2(\mathfrak{s}^{2}_{n,1},\theta, \mathbb{C}) = \Big\langle  \Big[\sum_{i=2}^{k}(-1)^{i}\Delta_{i+1,2k+2-i}\Big],\quad 2\leq k\leq \frac{n+1}{2}\Big\rangle, \quad \gamma = -2.$$

\begin{thm} \label{thm5.6} Let $\widetilde{L}$ be a one-dimensional extension of the solvable Lie algebra $\mathfrak{s}^{2}_{n,1},$  then $\widetilde{L}$ is isomorphic to the following algebras
$$\mathfrak{s}^{2}_{n+1,1}, \quad  \mathfrak{s}^{4}_{n+1,1}(0, 0, \dots, 0, 1), \quad \tau_{n+1,1}^3(0, 0, \dots, 0).$$
%
\end{thm}

\begin{proof}
Let \textbf{ $\gamma=-1$}. Denote by $\nabla_1=[\Delta_{3,1}],\nabla_2=[\Delta_{n+1,2}],$ then 
 for any $\psi= \delta_1\nabla_1+\delta_2\nabla_2,$ we have the action of the automorphism group
on the subspace $\langle \psi\rangle$ as $\langle \delta_1^*\nabla_1+\delta_2^*\nabla_2\rangle,$
where $$\delta_1^*=\delta_1c_2,\quad \delta_2^*=\delta_2 b_1^{n-1}c_2.$$

It is easy to see that $\operatorname{Ann}(\psi^0)\cap Z(n_{n,1})=0$ if and only if  $\delta_2\neq 0.$
Let us consider the following cases:
\begin{itemize}
\item If $\delta_1=0$, then
$\delta_1^{*}=0$  and
we have the representative $\langle\nabla_2\rangle$ and obtain the algebra $\mathfrak{s}^{2}_{n+1,1}.$

\item If $\delta_1\neq 0$, then choosing $ b_1=\sqrt[n-1]{\frac{\delta_1}{\delta_2}}, c_2=1$, we have the representative $\langle\nabla_1+\nabla_2\rangle$ and obtain the algebra $\mathfrak{s}^{4}_{n+1,1}(0, 0, \dots, 0, 1).$
\end{itemize}


 \textbf{$\gamma=-2$.} Let us denote $$\nabla_{k-1}=\Bigl[\sum_{i=2}^{k}(-1)^{i}\Delta_{i+1,2k+2-i}\Bigr],\quad 2\leq k\leq \frac{n+1}{2}.$$

Then for any
$\psi=\langle\delta_1\nabla_1+\delta_2\nabla_2+\ldots+\delta_{\frac{n-1}{2}}\nabla_{\frac{n-1}{2}}\rangle$, we have the action of the automorphism group
on the subspace $\langle \psi\rangle$ as
$\langle\delta_1^*\nabla_1+\delta_2^*\nabla_2+\ldots+\delta_{\frac{n-1}{2}}^*\nabla_{\frac{n-1}{2}}\rangle$
with restriction
$$\delta_{k}^{*}=b_1^{2k-1}\left(\sum\limits_{i=1}^{\frac {n-1} 2 +1-k}(-1)^{i+1}\delta_{i-1+k}c_{i+1}^2+2\sum\limits_{j=k+1}^{\frac{n-1}{2}}\sum\limits_{i=1}^{j-k}(-1)^{i+1}\delta_j c_{i+1}c_{2j-i-(2k-3)}\right), \quad 1 \leq k \leq \frac {n-1} 2.$$

Note that $\operatorname{Ann}(\psi^0)\cap Z(n_{n,1})=0$ if and only if  $\delta_{\frac{n-1}{2}}\neq 0$.
 Then choosing
 $$b_1=1,\quad c_2=\sqrt{\frac{1}{\delta_{\frac{n-1}{2}}}},$$
$$c_{2t}=-\frac{1}{2\delta_{\frac{n-1}{2}}c_2}\left(\sum\limits_{i=1}^{t}(-1)^{i+1}\delta_{\frac{n-1}{2}-t+i}c_{i+1}^2+2\sum\limits_{j=\frac{n+3}{2}-t}^{\frac{n-3}{2}}
  \sum\limits_{i=2}^{j+t-\frac{n-1}{2}}(-1)^{i}\delta_{j}c_{i}c_{2(j+t)-n+3-i}-2\delta_{\frac{n-1}{2}}\sum\limits_{i=2}^{t-1}(-1)^{i}c_{i+1}c_{2t+1-i}\right)$$
  where $2\leq t \leq\frac{n-3}{2}$, we have the representative $\langle\nabla_{\frac{n-1}{2}}\rangle$ and obtain  the algebra $\tau_{n+1,1}^3(0,0,\dots,0).$

\end{proof}

Similarly to the previous algebras we consider extensions of the algebras $\mathfrak{s}^{3}_{n,1}$ and $\mathfrak{s}^{4}_{n,1}(\alpha_3, \alpha_4, \dots, \alpha_{n-1}).$

\begin{prop} \label{DN5.6} Any element $\psi \in Z^2(\mathfrak{s}^{3}_{n,1},\theta, \mathbb{C})$ is formed by the following:

1) If $\gamma=-n,$ then
$$\begin{array}{llll} \psi(x, e_1)= b_{1,2}, & \psi(x, e_2)= b_{1,3}, & \psi(e_1, e_{n})= b_{2,n+1}, \\
 \psi(e_1, e_{j-1})= b_{2,j}, & \psi(x, e_j)= (j-1-n)b_{2,j},&3\leq j\leq n.
 \end{array}$$

2) If $\gamma=1-2k$ for some $k \ (2\leq k< \Big\lfloor\frac{n+1}{2}\Big\rfloor)$, then
$$\begin{array}{llll} \psi(x, e_1)= b_{1,2}, & \psi(x, e_2)= b_{1,3},\\
 \psi(e_1, e_j)= b_{2,j}, & 3\leq j\leq n,\\
  \psi(x, e_{2k})= b_{3,2k}, &  \psi(x, e_j)= (j-2k)b_{2,j},&3\leq j\leq n, \quad j \neq 2k ,\\   \psi (e_i,e_{2k+1-i})=(-1)^ib_{3,2k}, & 2 \leq i \leq k.
\end{array}$$

3) If $\gamma\neq 1-2k$, for any  $k \ ( 2\leq k\leq  \Big\lfloor\frac{n+1}{2}\Big\rfloor)$,
$$\begin{array}{llll} \psi(x, e_1)= b_{1,2}, & \psi(x, e_2)= b_{1,3},\\
 \psi(e_1, e_{j-1})= b_{2,j}, & \psi(x, e_j)= (j-1+\gamma)b_{2,j},&3\leq j\leq n.
\end{array}$$
\end{prop}

\begin{prop} Any element $\psi \in Z^2(\mathfrak{s}^{4}_{n,1}(\alpha_3, \alpha_4, \dots, \alpha_{n-1}),\theta, \mathbb{C})$ is formed by the following:

1) If  $\gamma =-1$, then
$$\begin{array}{llll} \psi(x, e_1)= b_{1,2}, & \psi(x, e_2)= b_{1,3}, &  \psi(e_1, e_{n-1})= b_{2,n},&\psi(e_1, e_{n})= b_{2,n+1}, \\
 \psi(e_1, e_{j-1})= b_{2,j}, & \psi(x, e_j)= \sum\limits_{k=j+1}^{n}\alpha_{k+2-j}b_{2,k+1},&3\leq j\leq n-1.
 \end{array}$$

 2) If  $\gamma =-2$ and $\alpha_{2t-1} =0$ for any $t \ (2 \leq t \leq \Big\lfloor\frac{n-1}{2}\Big\rfloor)$, then
$$\begin{array}{llll} \psi(x, e_1)= b_{1,2}, & \psi(x, e_2)= b_{1,3}, \\
 \psi(e_1, e_{j-1})= b_{2,j}, & \psi(x, e_j)=-b_{2,j}+ \sum\limits_{k=j+1}^{n-1}\alpha_{k+2-j}b_{2,k+1},&3\leq j\leq n,\\
 \psi(e_i,e_{n+2-i})=(-1)^ib_{3,n+1},& 2 \leq i \leq\Big \lfloor\frac{n+3} 2\Big\rfloor.
 \end{array}$$
  Note that in case of $n$ is even $b_{3,n+1}=0.$

 3) If  $\gamma =-2$ and $\alpha_{2t-1} \neq 0$ for some  $t \ ( 2 \leq t \leq \Big\lfloor\frac{n-1}{2}\Big\rfloor)$, then
$$\begin{array}{llll} \psi(x, e_1)= b_{1,2}, & \psi(x, e_2)= b_{1,3}, \\
 \psi(e_1, e_{j-1})= b_{2,j}, & \psi(x, e_j)=-b_{2,j}+ \sum\limits_{k=j+1}^{n-1}\alpha_{k+2-j}b_{2,k+1},&3\leq j\leq n.
 \end{array}$$

4) If  $\gamma \neq-1$ and   $\gamma \neq-2$,then
$$\begin{array}{llll} \psi(x, e_1)= b_{1,2}, & \psi(x, e_2)= b_{1,3}, \\
 \psi(e_1, e_{j-1})= b_{2,j}, & \psi(x, e_j)=(1+\gamma)b_{2,j}+ \sum\limits_{k=j+1}^{n-1}\alpha_{k+2-j}b_{2,j+1},&3\leq j\leq n.\\
 \end{array}$$

\end{prop}

It is not difficult to obtain that 2-coboundaries with respect to $\theta$ of this algebras are

$$B^2(\mathfrak{s}^{3}_{n,1},\theta, \mathbb{C}) = \left\{\begin{array}{llll} df(x,e_1)=-(1+\gamma )c_1-c_2, \\
df(x,e_i)=-(i-1+\gamma)c_{i}, & 2\leq i\leq n, \\
df(e_1,e_i)=-c_{i+1}, & 2\leq i\leq n-1. \end{array}\right.$$

$$B^2(\mathfrak{s}^{4}_{n,1}(\alpha_3, \alpha_4, \dots, \alpha_{n-1}),\theta, \mathbb{C}) = \left\{\begin{array}{llll} df(x,e_1)=-\gamma c_1, \\
df(x,e_i)=-(1+\gamma)c_{i}-\sum\limits_{l=i+2}^{n}\alpha_{l+1-i}c_l, & 2\leq i\leq n, \\
df(e_1,e_i)=-c_{i+1}, & 2\leq i\leq n-1. \end{array}\right.$$

\begin{thm} \label{thm5.3} Any non-split extension of the solvable Lie algebra $\mathfrak{s}^{3}_{n,1}$  is isomorphic to the algebra
$\mathfrak{s}^{3}_{n+1,1}.$
\end{thm}

\begin{proof}
Note that in cases of
$\gamma \neq -n,$ we have $\operatorname{Ann}(\psi^0)\cap Z(n_{n,1})=\{e_n\} \neq 0.$ Therefore, to get a non-split extension of the solvable Lie algebra $\mathfrak{s}^{3}_{n,1}$ it is enough to consider the case of $\gamma=-n.$ In this  case we have $\dim Z^2(\mathfrak{s}^{3}_{n,1},\theta, \mathbb{C})=n+1,$ $\dim B^2(S_{n+1,3},\theta, \mathbb{C})=n$
which implies $\dim H^2(\mathfrak{s}^{3}_{n,1},\theta, \mathbb{C}) =1$
 and a basis of $H^2(\mathfrak{s}^{3}_{n,1},\theta, \mathbb{C})$ is formed by the following cocycle
 $$H^2(\mathfrak{s}^{3}_{n,1},\theta,\mathbb{C}) =\langle[\psi]\rangle, \quad \psi(e_n, e_1) = e_{n+1}, \quad \gamma=-n.$$

New products of the extension algebra $\mathfrak{s}^{3}_{n,1}\oplus\{e_{n+1}\}$ is
    $$[e_n,e_1]=\psi(e_n,e_1) =  e_{n+1},\quad [e_{n+1}, x]=-\theta(x)e_{n+1}=ne_{n+1} $$
and we obtain the algebra $\mathfrak{s}^{3}_{n+1,1}.$
\end{proof}

\begin{thm} One-dimensional extensions of the solvable Lie algebra $S_{n+1,4}(\alpha_3,\alpha_4,\dots,\alpha_{n-1})$  are the algebras
$\mathfrak{s}^{4}_{n+1,1}(\alpha_3, \alpha_4, \dots, \alpha_{n-1}, \alpha_{n})$
and
$\tau_{n+1,1}^3(\alpha_4, \alpha_6, \dots \alpha_{n-1}).$

\end{thm}

\begin{proof}

It is not difficult to see that to get a non-split extension of the solvable Lie algebra $\mathfrak{s}^{4}_{n,1}(\alpha_3, \alpha_4, \dots, \alpha_{n-1})$ it is enough to consider the cases
 $\gamma=-1$ and $\gamma=-2,$ $n$ is odd, $\alpha_{2t-1} =0$ for any $t \ (2 \leq t \leq \frac{n-1}{2})$ .
In these cases we have the followings
$$\dim B^2(\mathfrak{s}^{4}_{n,1}(\alpha_3, \alpha_4, \dots, \alpha_{n-1}),\theta,\mathbb{C})  = \left\{\begin{array}{llll} n-1 & \text{if} & \gamma=-1,\\
n & \text{if} & \gamma = -2, \end{array}\right.$$
and
$$H^2(\mathfrak{s}^{4}_{n,1}(\alpha_3, \alpha_4, \dots, \alpha_{n-1}),\theta,\mathbb{C})  = \left\{\begin{array}{llll} \Big\langle  [\Delta_{3,1}], [\Delta_{n+1,2}-\sum\limits_{i=4}^{n}\alpha_{n+3-i}\Delta_{i,1}]\Big\rangle & \gamma=-1,\\[5mm]
\Big\langle  \Big[\sum\limits_{i=3}^{\lfloor\frac{n+1}{2}\rfloor}(-1)^{i}\Delta_{i,n+4-i}\Big]\Big\rangle & \gamma = -2. \end{array}\right.$$

Thus, we consider following cases

(1) Let $\gamma=-1$.  Denote by
$$ \nabla_1=[\Delta_{3,1}], \quad \nabla_2=\Bigl[\Delta_{n+1,2}-{\sum_{i=4}^{n}\alpha_{n+3-i}\Delta_{i,1}}\Bigl]. $$

Then any $\psi=\langle \delta_1\nabla_1+\delta_2\nabla_2\rangle,$ we have the action of the automorphism group
on the subspace $\langle \psi\rangle$ as $\langle \delta_1^*\nabla_1+\delta_2^*\nabla_2\rangle,$
where $$\delta_1^*=\delta_1c_2,\quad \delta_2^*=\delta_2c_2.$$

It is easy to see that $\operatorname{Ann}(\psi^0)\cap Z(n_{n,1})=0$ if and only if  $\delta_2\neq 0.$
Thus, we have the representative $\langle\alpha_{n}\nabla_1+\nabla_2\rangle$ and obtain the algebra $\mathfrak{s}^{4}_{n+1,1}(\alpha_3, \alpha_4, \dots, \alpha_{n-1}, \alpha_{n}).$

(2) Let  $\gamma=-2.$ Then we have the orbit
$\Big\langle \Big[ \sum\limits_{i=3}^{\lfloor\frac{n+1}{2}\rfloor}(-1)^{i}\Delta_{i,n+4-i}\Big]\Big\rangle$
and obtain the algebra $\tau_{n+1,1}^3(\alpha_4, \alpha_6, \dots \alpha_{n-1}).$

\end{proof}

\section{Conclusion}

In this work we obtain all one-dimensional non-split central extension of the algebra
$n_{n,1}$ and solvable Lie algebras with nilradical $n_{n,1}.$

As it was shown in Section 3, the dimension of extension of the algebra
$n_{n,1}$ may be up to  $\Big\lfloor\frac {n+1} 2 \Big \rfloor$ and one-dimensional extensions are $n_{n+1,1},$ $Q_{n+1}$ and $L_k(2\leq k\leq \Big\lfloor \frac{n}{2}\Big\rfloor).$

We know that solvable Lie algebras with nilradicals $n_{n+1,1},$ $Q_{n+1}$  are
$$\mathfrak{s}^{1}_{n+1,1}(\beta), \quad  \mathfrak{s}^{2}_{n+1,1}, \quad \mathfrak{s}^{3}_{n+1,1}, \quad
\mathfrak{s}^{4}_{n+1,1}(\alpha_3, \alpha_4, \dots, \alpha_{n-1}), \quad
\mathfrak{s}_{n+1,2},$$
$$\tau_{n+1,1}^1(\alpha), \quad \tau_{n+1,1}^2, \quad \tau_{2n,1}^3(\alpha_4, \alpha_6, \dots \alpha_{2n-2}), \quad \tau_{n+1,2}.$$

Moreover, it is not difficult to obtain that codimension of the solvable Lie algebra with nilradical $L_k$ is equal to one and this solvable algebra is isomorphic to the algebra
$\widetilde{L}_k$ for any $k \ (2\leq k\leq \Big\lfloor \frac{n}{2}\Big\rfloor).$

Therefore, we have the following commute diagrams
$$n_{n,1}\xrightarrow{codim =2} \mathfrak{s}_{n,2} \xrightarrow{extension=1}\begin{array}{c}\mathfrak{s}_{n+1,2}, \\ \tau_{n+1,2}.\end{array}$$

$$n_{n,1}\xrightarrow{extension=1} \begin{array}{lcl} n_{n+1,1} & \xrightarrow{codim =2} & \mathfrak{s}_{n+1,2}, \\ Q_{n+1} & \xrightarrow{codim =2} & \tau_{n+1,2}, \\ L_k & \xrightarrow{codim =2} & \text{no algebra}. \end{array}$$

and

$$n_{n,1}\xrightarrow{codim =1} \begin{array}{ccl}
\mathfrak{s}^{1}_{n,1}(\beta) &\xrightarrow{extension=1} & \left\{\begin{array}{l} \mathfrak{s}^{1}_{n+1,1}(\beta), \\ \tau_{n+1,1}^1(\beta), \\ \tau_{n+1,1}^2, \\  \widetilde{L_k},\end{array}\right.\\[7mm]
\mathfrak{s}^{2}_{n,1} &\xrightarrow{extension=1} & \left\{\begin{array}{l}\mathfrak{s}^{2}_{n+1,1}, \\ \mathfrak{s}^{4}_{n+1,1}(0, 0, \dots, 0, 1), \\ \tau_{n+1,1}^3(0, 0, \dots, 0),\end{array}\right.\\[7mm]
\mathfrak{s}^{3}_{n,1} &\xrightarrow{extension=1} & \mathfrak{s}^{3}_{n+1,1},\\[2mm]
\mathfrak{s}^{4}_{n,1}(\alpha_3, \alpha_4, \dots, \alpha_{n-1}) &\xrightarrow{extension=1} & \left\{\begin{array}{l}\mathfrak{s}^{4}_{n+1,1}(\alpha_3, \alpha_4, \dots, \alpha_{n-1}, \alpha_{n}), \\ \tau_{n+1,1}^3(\alpha_4, \alpha_6, \dots \alpha_{n-1}).\end{array}\right.
 \end{array}$$

\

$$n_{n,1}\xrightarrow{extension=1} \begin{array}{lcl} n_{n+1,1} & \xrightarrow{codim =1} & \mathfrak{s}^{1}_{n+1,1}(\beta), \
\mathfrak{s}^{2}_{n+1,1}, \
\mathfrak{s}^{3}_{n+1,1} \
\mathfrak{s}^{4}_{n+1,1}(\alpha_3, \alpha_4, \dots, \alpha_{n-1}), \\ Q_{n+1} & \xrightarrow{codim =1} & \tau_{n+1,1}^1(\alpha), \ \tau_{n+1,1}^2,
\ \tau_{n+1,1}^3(\alpha_4, \alpha_6, \dots \alpha_{n-1}), \\ L_k & \xrightarrow{codim =1} & \widetilde{L_k}. \end{array}$$

\end{document}